\documentclass[UTF8,a4paper,12pt]{elsarticle}
\usepackage[left=2.50cm, right=2.50cm, top=2.50cm, bottom=2.50cm]{geometry} %页边距

\usepackage{latexsym}
\usepackage{amssymb}
\usepackage{newlfont}
\usepackage{hyperref}
\usepackage{mathrsfs}
\usepackage{amsmath}
\usepackage{amsthm}
\usepackage{booktabs}
\usepackage{amsfonts}
\usepackage{mathrsfs}

\newtheorem{assumption}{Assumption}

\newtheorem{thm}{Theorem}[section]
\newtheorem{cor}[thm]{Corollary}
\newtheorem{prop}[thm]{Proposition}
\newtheorem{Def}[thm]{Definition}
\newtheorem{lemma}[thm]{Lemma}
\newtheorem{remark}[thm]{Remark}
\numberwithin{equation}{section}
\newcommand{\n}{\Vert}
\newcommand{\h}{\mathcal{H}}
\newcommand{\e}{\mathcal{E}}

\begin{document}

\begin{frontmatter}
	\title{\textbf{Well-posedness and strong attractors for a beam model with degenerate nonlocal strong damping}}
	\author{ Senlin Yan, Chengkui Zhong }
	
	\address{Department of Mathematics, Nanjing University, Nanjing, 210093, PR China}
	
	\cortext[]{
		E-mails: dg20210019@smail.nju.edu.cn(S. Yan), ckzhong@nju.edu.cn(C. Zhong), }

	\begin{abstract}
		This paper is devoted to initial-boundary value problem of an extensible beam equation with degenerate nonlocal energy damping in $\Omega\subset\mathbb{R}^n$: $u_{tt}-\kappa\Delta u+\Delta^2u-\gamma(\n \Delta u\n^2+\n u_t\n^2)^q\Delta u_t+f(u)=0$. We prove the global existence and uniqueness of weak solutions, which gives a positive answer to an open question in \cite{Silva}. Moreover, we establish the existence of a strong attractor for the corresponding weak solution semigroup, where the ``strong" means that the compactness and attractiveness of the attractor are in the topology of a stronger space $\mathcal{H}_{\frac{1}{q}}$.
	\end{abstract}
	
	
\end{frontmatter}

\section{Introduction}
    In this paper, we study the following extensible beam equation with degenerate nonlocal energy damping
    \begin{equation}\label{1-1}
    	u_{tt}-\kappa\Delta u+\Delta^2u-\gamma(\n \Delta u\n^2+\n u_t\n^2)^q\Delta u_t+f(u)=0\quad \mathrm{in}\quad \Omega\times\mathbb{R}^+,
    \end{equation}
    where $\Omega\subset\mathbb{R}^n$ is a bounded $C^\infty$-domain; $\kappa\geq0,\ \gamma>0$ and $q\geq1$ are constants; $\n\cdot\n$ denotes the $L^2(\Omega)-$norm; the assumptions on the nonlinearity $f(u)$ will be given later.
    
    The initial conditions associated with (\ref{1-1}) are given by 
    \begin{equation}\label{1-2}
    	u(x,0)=u_0(x),\quad u_t(x,0)=u_1(x),\quad x\in\Omega,
    \end{equation}
    and the following hinged boundary condition is considered:
    \begin{align}
    	\label{1-4}u|_{\partial\Omega}=\Delta u|_{\partial\Omega}=0,\quad t\in\mathbb{R}^+.
    \end{align}
    
    In 1989 Balakrishnan and Taylor \cite{Taylor} proposed the following one-dimensional beam model with nonlocal energy damping in flight structures
    \begin{equation}\label{1-5}
    	u_{tt}-2\zeta\sqrt{\lambda}u_{xx}+\lambda u_{xxxx}-\gamma\Big[\int_{-L}^{L}(\lambda|u_{xx}|^2+|u_t|^2)dx\Big]^qu_{xxt}=0,
    \end{equation}
    where $u=u(x,t)$ represents the transversal deflection of a beam with length $2L>0$ in the rest position, $\gamma>0$ is a damping coefficient, $\zeta$ is a constant appearing in the approximation of Krylov-Bogoliubov and $\lambda=\frac{2\zeta\omega}{\sigma^2}$ with $\omega$ being the model frequency and $\sigma^2$ the spectral density of a Gaussian external force (see Eq. (4.2) in \cite{Taylor}).  After that, Silva, Narciso and Vicente \cite{Silva} investigated $n$-dimensitional version of (\ref{1-5}) with nonlinear source term 
    \begin{equation}\label{1-6}
    	u_{tt}-\kappa\Delta u+\Delta^2u-\gamma(\n \Delta u\n^2+\n u_t\n^2)^q\Delta u_t+f(u)=0,
    \end{equation}
    subjected to clamped boundary condition. They proved existence and polynomical stability of regular solutions in the phase space $\h_2=(H^4\cap H^2_0)\times H_0^2$ when $q\geq1$ and $f(u)$ satisfies the growth condition
    \[|f'(u)|\leq C(1+|u|^\rho),\quad\mathrm{with}\quad 0\leq\rho\leq\frac{4}{n-4}\]
    and the dissipative condition 
    \begin{equation}\label{1-7}
    	-\frac{\theta}{2}|u|^2\leq F(u)\leq f(u)u+\frac{\theta}{2}|u|^2,
    \end{equation}
    where $\theta\in[0,\lambda_1)$ and $\lambda_1$ is the first engenvalue of the bi-harmonic operator $\Delta^2$. Due to the degeneracy of the damping coefficient, they only deal with regular solution and make the existence and uniqueness of weak solution an open question (see \cite{Silva}, Remark 2). Later, by neglecting the nonlocal term corresponding to the velocity, Cavalcanti et al. \cite{Cavalcanti} studied the model 
    \begin{equation}\label{1-9}
    	u_{tt}+\Delta^2u-M(\n\nabla u\n^2)\Delta u+\n \Delta u\n^2A u_t=0,
    \end{equation}
    with $A=-\Delta$ or $A=I$ (identity) and the clamped boundary condition. Their main result states that the energy $E_u(t)=\frac{1}{2}\n\Delta u(t)\n^2+\frac{1}{2}\n u_t(t)\n^2+\frac{1}{2}\int_0^{\n\nabla u(t)\n^2}M(s)ds$ decays to zero uniformly for every regular solution whose initial data is taken from a bounded subset of $\h_2$. In particular, when $A=-\Delta$, they only obtained the existence and uniqueness for regular solution. 
    We also refer to \cite{Cavalcanti1,Clark1,Clark2,Gomes1,Gomes2,Ma1,Ma2,Yang} for the pioneering studies where well-posedness, stability and asymptotic behavior of solutions are studied for beam/plate models with degenerate or non-degenerate nonlocal dampings.

    The global attractor theory of beam equations with nonlocal damping has received much more attention in recent years (cf. \cite{Chueshov2015,Chueshov2008,Li1,Silva2014,Silva2015,Silva2017,Sun1,Sun2,Zhao1,Zhao2} and reference therein). Motivated by model (\ref{1-6}), the authors in \cite{Sun2} studied a more general case
    \begin{equation}\label{1-8}
    	u_{tt}+\Delta^2u-\kappa\phi(\n\nabla u\n^2)\Delta u-M(\n\xi_u\n_\h^2)\Delta u_t+f(u)=h,
    \end{equation}
    where $\n\xi_u\n_\h^2=\n \Delta u\n^2+\n u_t\n^2$ and $M
    \in C^1(\mathbb{R}^+)$. Taking $M(s)=\gamma s^q$, the energy damping in (\ref{1-8}) reduces to the one in (\ref{1-6}). But they restrict themselves to the non-degenerate case: $M(s)>0$ for all $s\geq0$, in order to verify the existence and uniqueness of weak solutions. Finally, they obtained the strong global and exponential attractors and their robustness on the perturbed parameter $\kappa\in\Lambda$, where the ``strong'' means that the compactness, the attractiveness and the finiteness of the fractal dimension of the attractors are all in the topology of the regular space $\h_2$. Taking into account the rotational force in beam equations, the authors in \cite{Sun1} considered the following model
    \begin{equation}
    	(1-\alpha\Delta)u_{tt}+\Delta^2u-\phi(\n\nabla u\n^2)\Delta u-M(\n\xi_u\n_\h^2)\Delta u_t+f(u)=h,
    \end{equation}
    where  $\n\xi_u\n_\h^2=\n \Delta u\n^2+\n u_t\n^2$ and $M$ is still non-degenerate: $M(s)>0,\forall s\geq0$. They proved the well-posedness of weak solutions, the existence of global and exponential attractors and their continuity with respect to $\alpha\in[0,1]$.
    
     To our best knowledge, the existence and uniqueness of weak solution of problem (\ref{1-1}) is still open and the main difficulty of which comes from degenerate strong energy damping, as mentioned in \citep[Remark 2]{Silva}. Specifically, for the non-degenerate case, one can obtain the well-posedness of weak solution by approximating regular solutions in a standard way (cf. \cite{Sun1,Sun2}). When considering nonlocal weak damping such as $(\n \Delta u\n^2+\n u_t\n^2)^qu_t$, one can use the semigroup theory to obtain mild solutions (cf. \cite{Cavalcanti,Gomes1,Gomes2}), since $(\n \Delta u\n^2+\n u_t\n^2)^qu_t$ is locally Lipschitz continuous on $\h$. However, in the case of degenerate strong damping, both approximation method and semigroup theory are invalid.
    
    Our goal in this paper is to overcome this difficulty and analyze the well-poseness and long-time behavior of weak solution for problem (\ref{1-1}). Using the interpolation inequality and Galerkin approximation, we first prove that for any initial data $(u_0,u_1)\in \h$, the weak solution exists and it possesses an extra regularity when $t>0$, see Theorem \ref{well-posedness}. To prove the uniqueness, we establish a lemma which implies that any trajectory with non-zero initial data will not decay to zero in any finite time, although it will decay to zero as $t\rightarrow\infty$, see Lemma \ref{lemma1}. Based on this lemma, the damping coefficient is bounded from below in any finite time interval and the uniqueness is reduced to non-degenerated case, see Theorem \ref{uniqueness}. Theorem \ref{well-posedness} and Theorem \ref{uniqueness} can be seen as our first result and they have answered the open question proposed in \citep[Remark 2]{Silva}.
    
    On the other hand, in order to study the global attractor of weak solution semigroup, the nonlinearity $f(u)$ is taken more generally than (\ref{1-7}), see Remark \ref{remark3.2}. It is worth mentioning that, under the condition (\ref{1-7}), the results in \cite{Silva} implies that the solution semigroup has the global attractor as a set of a single point: $\mathcal{A}=\{(0,0)\in\h\}$. In our case, we prove the existence of the global attractor for the dynamical system $(\h,S(t))$ and it is a bounded set in a more regular space $\h_{\frac{1}{q}}$ due to the extra regularity of weak solution. Moreover, using the norm-to-weak continuous semigroup theory established in \cite{Zhong}, we show that $S(t)$ is actually a norm-to-weak continuous $(\h,\h_{\frac{1}{q}})$-semigroup and the attractor $\mathcal{A}$ is actually an $(\h,\h_{\frac{1}{q}})$-global attractor, see Theorem \ref{thm}. This is our second result.
    
    The rest of this paper is organized as follows. In Section 2, we introduce some functional spaces and give an abstract formulation of problem (\ref{1-1}). In Section 3, we discuss the existence and uniqueness of weak solutions. In Section 4, we prove the existence of the global attractor for the weak solution semigroup. Finally, the existence of an $(\h,\h_{\frac{1}{q}})$-global attractor is established in Section 5.

\section{Preliminaries}
    In this section, we recall the theory on functional spaces that will be used later and give an abstract formulation of problem (\ref{1-1}). For brevity, we use the following abbreviations:
    \[L^p=L^p(\Omega),\quad \n\cdot\n=\n\cdot\n_{L^2},\]
    with $p\geq1$ and $(\cdot,\cdot)$ stands for the $L^2$-inner product as well as the notation of duality pairing between dual spaces.
    
    Let $\Omega$ be a $C^\infty$-domain of $\mathbb{R}^n$. For $0\leq s<\infty,\ 1<p<\infty$, we denote by $W^{s,p}:=W^{s,p}(\Omega)$ the Sobolev-Slobodeckij spaces over $\Omega$, which are defined as restrictions of the corresponding spaces over $\mathbb{R}^n$ (see \cite{Triebel} for more details). Note that, in the sense of equivalent norm, the spaces $W^{s,p}$ with $s=0,1,2,\cdots$ coincide with the classical Sobolev spaces of distributions whose derivatives up to order $s$ belong to $L^p$. The closure of $C_c^\infty(\Omega)$ in the space $W^{s,p}$ is denoted by $W_0^{s,p}$. In the particular case $p=2$ we use the notations $H^s:=W^{s,2}$ and $H_0^s:=W_0^{s,2}$. We also recall the classical Sobolev embedding theorem
    \begin{equation}\label{2-1}
    	H^s\hookrightarrow L^q,\qquad \frac{1}{q}\geq \frac{1}{2}-\frac{s}{n},\quad s\geq0,\quad 2\leq q<\infty.
    \end{equation}
     
     Let $A=\Delta^2$. As we know, the operator $A$ associated with the boundary condition (\ref{1-4}) is a self-adjoint positive operator acting on $L^2$ and possesses discrete spectrum: 
     \[Ae_i=\lambda_i e_i,\quad 0<\lambda_1\leq\lambda_2\leq\cdots,\quad \lim_{i\rightarrow\infty}\lambda_i=\infty,\]
     where the eigenvectors $\{e_i\}_{i\geq1}$ are chosen to be an orthonormal basis in $L^2$. Then, we define the powers of $A$ as follows:
     \[A^su=\sum_{i=1}^{\infty}c_i\lambda_i^se_i,\quad s\in\mathbb{R},\quad u=\sum_{i=1}^{\infty}c_ie_i\in D(A^s),\]
     where $D(A^s)=\{u|u\in L^2,\ A^su\in L^2\}$ is a Hilbert space with the scalar product and the norm
     \[(u,v)_{D(A^s)}=(A^su,A^sv),\quad \n u\n_{D(A^s)}=\n A^su\n.\]
     In particular, $D(A)=\{u|u\in H^4,\ u|_{\partial\Omega}=\Delta u|_{\partial\Omega}=0\}$, $D(A^\frac{1}{2})=H^2\cap H^1_0$, $D(A^\frac{1}{4})=H^1_0$ and $D(A^s)\hookrightarrow H^{4s}$ for all $s\in[0,1]$.
     
     We have the compact embedding 
     \begin{equation}\label{2-2}
     	D(A^\alpha)\hookrightarrow\hookrightarrow D(A^\beta),\quad\forall \alpha>\beta.
     \end{equation}
     and the interpolation inequality:
     \begin{equation}\label{2-3}
     	\n A^su\n\leq\n A^ru\n^\theta\n A^tu\n^{1-\theta},\quad\forall u\in D(A^r)\cap D(A^t),
     \end{equation}
     where $-\infty<r\leq s\leq t<\infty,\ \theta\in[0,1]$ and $s=\theta r+(1-\theta)t$. Moreover, $D(A^{-s})$ can be viewed as the dual space of $D(A^s)$ for every $s\in\mathbb{R}$.

     For convenience, we also introduce the following notations that will be used throughout the paper. Let the phase spaces
     \[V_s:=D(A^\frac{s}{4}),\quad \mathcal{H}_s:=V_{2+s}\times V_s,\quad s\in\mathbb{R}\]
     with the norm
     \[\n u\n_{s}:=\n u\n_{V_s},\quad \n(u,v)\n^2_{\mathcal{H}_s}:=\n u\n^2_{2+s}+\n v\n^2_{s},\]
     and denote $\h=\h_0$ when $s=0$ for simplicity. From the above definition, we have $V_s\hookrightarrow H^{s}$ for any $s\in[0,4]$. We also have the inequalities
     \begin{equation}\label{2-4}
     	\lambda_1\n u\n^2\leq \n  u\n_2^2,\quad \lambda_1^\frac{1}{2}\n  u\n_1
     	^2\leq \n  u\n_2^2,\quad\forall u\in V_2.
     \end{equation}

     Using above notations, we can rewrite problem (\ref{1-1})-(\ref{1-4}) at an abstract level given by
     \begin{equation}\label{eq}
     	\begin{cases}
     		u_{tt}+\kappa A^\frac{1}{2}u+Au+\gamma\n(u,u_t)\n_{\h}^{2q}A^\frac{1}{2}u_t+f(u)=0,\\
     		(u(0),u_t(0))=(u_0,u_1),\\
     		u|_{\partial\Omega}=\Delta u|_{\partial\Omega}=0.
     	\end{cases}
     \end{equation}

\section{Well-posedness of weak solutions}
     In this section, we study the existence and uniqueness of weak solutions of problem (\ref{eq}).
     \begin{assumption}\label{assumption}
     	(i) the constants $\kappa\geq0$, $\gamma >0$, $q\geq1$;
     	
     	(ii) $f\in C^1(\mathbb{R})$, $f(0)=0$,
     	\begin{align}\label{3-1}
     		|f'(u)|\leq C(1+|u|^p),\ \forall u\in\mathbb{R},
     	\end{align}
        \begin{align}\label{3-2}
        	\liminf_{|u|\rightarrow\infty} f'(u)>-\lambda_1,\ \forall u\in\mathbb{R},
        \end{align}
        where $C>0$ and the growth exponent $p$ satisfies
        \[1\leq p<\infty\ \ \mathrm{if}\ \ 1\leq n\leq4\quad and\quad 1\leq p\leq\frac{4}{n-4}\ \ \mathrm{if}\ \ n\geq5.\]
     \end{assumption}
  
     \begin{remark}
     	Assumption (\ref{3-1}) and the Mean Value Theorem imply that there exists some constant $C>0$ such that
     	\begin{equation}\label{3-3}
     		|f(u)|=|f(u)-f(0)|\leq C(1+|u|^p)|u|,\ \forall u\in\mathbb{R}.
     	\end{equation}
     	
     	Set $F(u)=\int_{0}^{u}f(\tau)d\tau$. Assumption (\ref{3-2}) implies that
     	\begin{equation}\label{3-4}
     		\int_\Omega F(u)dx\geq -\frac{\mu}{2}\n u\n^2-C,
     	\end{equation}
        \begin{equation}\label{3-5}
     		(f(u),u)\geq \int_{\Omega}F(u)dx-\frac{\mu}{2}\n u\n^2-C,
     	\end{equation} 
        for some $C>0$ and $\mu\in[0,\lambda_1)$ (See \cite{Pata1}).
     \end{remark}
 
     \begin{remark}\label{remark3.2}
     	The assumptions (\ref{3-1})-(\ref{3-2}) of $f$ are more general than assumptions (14)-(16) in \cite{Silva}. That is to say, if a function $f$ satisfies (14)-(16) in \cite{Silva} then $f$ also satisfies (\ref{3-1})-(\ref{3-2}). While there exists some $f$ which satisfies (\ref{3-1})-(\ref{3-2}) and do not satisfy (14)-(16) in \cite{Silva}, e.g. $f(u)=u-(3\lambda_1+9)u^2+u^3$ when $n=5$.
     \end{remark}
 
     \begin{Def}
     	For any $T>0$, a function $u(t),\ t\in[0,T]$ is said to be a weak solution of problem (\ref{eq}) if $(u,u_t)\in L^\infty(0,T;\h)$ and Eq. (\ref{eq}) is satisfied in the sense of distribution, i.e.
     	\begin{align*}
     		-\int_{0}^{T}(u_t,\phi_t)dt+&\kappa\int_{0}^{T}(A^\frac{1}{4} u,A^\frac{1}{4}\phi)dt+\int_{0}^{T}(A^\frac{1}{2} u,A^\frac{1}{2}\phi)dt\\
     		&+\gamma\int_{0}^{T}\n(u,u_t)\n_\h^{2q}(A^\frac{1}{4} u_t,A^\frac{1}{4}\phi)dt+\int_{0}^{T}(f(u),\phi)dt=0,
     	\end{align*}
     	for any $\phi\in C_c^\infty((0,T)\times\Omega)$.
     \end{Def}
     
     We restrict ourselves to the case $n\geq 5$ in this section, but all the conclusions in this section hold for $1\leq n\leq 4$.
     
     \begin{thm}\label{well-posedness}
     	Let $T>0$ be arbitrary and Assumption \ref{assumption} be valid. Then, we have:
     	
        (i) For any initial data $(u_0,u_1)\in\h$, problem (\ref{eq}) admits a weak solution $u$ and the solution possesses the following regularity
        \begin{equation}
        	(u,u_t)\in L^\infty(a,T;\h_s),\quad \forall 0<a<T,
        \end{equation}
        where $s=\frac{1}{q}$.
        
        (ii) The following estimates hold for the solution u:
        \begin{equation}\label{3-7}
        	\n(u(t),u_t(t))\n_{\h}^2+\int_{0}^{t}\n(u(\tau),u_t(\tau))\n_{\h}^{2q}\n u_t(\tau)\n_1^2d\tau\leq Q_1\big(\n(u_0,u_1)\n_\h\big),\ \forall t\in[0,T],
        \end{equation}
        \begin{equation}\label{3-8}
        	\begin{split}
        		\n(u(t),u_t(t))\n_{\h_s}^2+\int_{t}^{t'}\n(u(\tau),&u_t(\tau))\n_{\h}^{2q}\n u_t(\tau)\n_{1+s}^2d\tau\\
        		&\leq t^{-(1+s)}Q_2\big(\n(u_0,u_1)\n_\h,T\big),\ \forall 0<t\leq t'\leq T,
        	\end{split}
        \end{equation}
        where $s=\frac{1}{q}$ and $Q_1,Q_2$ are monotone increasing functions independent of $u$ and $t$.
        
        (iii) The solution $u$ satisfies the energy equality
        \begin{equation}\label{energy}
        	E((u(t)))+\gamma\int_{\tau}^{t}\n(u(s),u_t(s))\n_{\h}^{2q}\n u_t(s)\n_1^2ds=E(u(\tau)),\quad \forall0\leq \tau\leq t\leq T,
        \end{equation}
        where 
        \begin{equation}\label{energyfunction}
        	E(u)=\frac{1}{2}\n u_t\n^2+\frac{\kappa}{2}\n  u\n_1^2+\frac{1}{2}\n  u\n_2^2+\int_{\Omega}F(u)dx.
        \end{equation}
        In particilar, $(u,u_t)\in C([0,T];\h)$.
     \end{thm}
     \begin{proof}
     	\textbf{(i)} We use Faedo-Galerkin method to prove the existence of weak solutions. Let $0<\lambda_1\leq\lambda_2\leq\cdots$ be the eigenvalues of the operator $A=\Delta^2$ with boundary condition (\ref{1-3}) or (\ref{1-4}) and $e_1,e_2,\cdots$ be the corresponding eigenfunctions such that they form an orthonormal basis in $L^2$. Then, $e_j\in C^\infty(\Omega)$ for all $j\geq1$ since $\Omega$ is smooth. Let $P_n:L^2\rightarrow L^2$ be the orthoprojector to the subspace $\mathrm{span}\{e_1,\cdots,e_n\}$. Consider the following approximate problem
     	\begin{equation}\label{ode}
     		\begin{cases}
     			(u^n_{tt},e_j)+\kappa(A^\frac{1}{2}u^n,e_j)+(Au^n,e_j)
     			+\gamma\n(u^n,u_t^n)\n_{\h}^{2q}(A^\frac{1}{2}u_t,e_j)+(f(u^n),e_j)=0,\\
     			(u^n(0),u^n_t(0))=(u_0^n,u_1^n):=(P_nu_0,P_nu_1)\rightarrow(u_0,u_1)\ \mathrm{in}\ \h,\quad j=1,2,\cdots,n,
     		\end{cases}
     	\end{equation}
     	which has a local solution 
     	\[u^n(t)=\sum_{j=1}^{n}T^{jn}(t)e_j\in\mathrm{Span}\{e_1,\cdots,e_n\},\ t\in[0,T_n)\]
     	by ODE theory. We need to give some a priori estimates in order to extend the local solution to the whole interval $[0,T]$ and to conclude the existence of weak solutions of problem (\ref{eq}).
     	
     	\textbf{A priori estimate.} Suppose that $\n(u_0,u_1)\n_{\h}\leq R.$ Then $\n(u^n_0,u^n_1)\n_{\h}\leq R,\ \forall n\in\mathbb{N}$. Multiplying (\ref{ode}) by $T_t^{jn}$ for every $j=1,2,\cdots,n$ and sum up all the equations, we get
     	\begin{equation}\label{3-2-2}
     		\frac{d}{dt}E(u^n)+\gamma\n(u^n,u^n_t)\n_{\h}^{2q}\n u^n_t\n_1^2=0,
     	\end{equation}
     	where $E(u^n)$ is the energy functional (\ref{energyfunction}) for $u^n$. From (\ref{2-4}) and (\ref{3-4}), we have
     	\begin{equation}\label{3-2-1}
     		E(u^n)\geq\frac{1}{2}\n u^n_t\n^2+\frac{1}{2}\n u^n\n_2^2-\frac{\mu}{2\lambda_1}\n u^n\n_2^2-C\geq\frac{\lambda_1-\mu}{2\lambda_1}\n(u^n,u^n_t)\n_{\h}^2-C.
     	\end{equation}
     	Integrating (\ref{3-2-2}) over $[0,t]$ yields
     	\begin{equation}\label{3-2-3}
     		E(u^n(t))+\gamma\int_{0}^{t}\n(u^n(\tau),u^n_t\tau)\n_{\h}^{2q}\n u^n_t(\tau)\n_1^2d\tau= E(u^n(0)).
     	\end{equation}
     	Combining (\ref{3-2-1}) and (\ref{3-2-3}), using (\ref{3-3}), we obtain that
     	\begin{equation}\label{3-2-4}
     		\n(u^n(t),u^n_t(t))\n_{\h}^2+\int_{0}^{t}\n(u^n(\tau),u^n_t(\tau))\n_{\h}^{2q}\n u^n_t(\tau)\n_1^2d\tau\leq C_R
     	\end{equation}
     	holds for all $t\in[0,T_n)$. This estimate allows us to extend all local solutions $u^n(t)$ to the whole interval $[0,T]$ and the estimate (\ref{3-2-4}) holds true for all $t\in[0,T]$ and $n\in\mathbb{N}$. Then, by (\ref{3-3}) and the Sobolev embedding $V_2\hookrightarrow L^{2p+2}$, we can deduce from (\ref{3-2-4}) that
     	\begin{equation}\label{3-2-5}
     		\begin{split}
     			\n f(u^n)\n^2&\leq\int_{\Omega}C(1+|u|^p)^2|u|^2dx\\
     			&\leq C\n u\n^2+C\n u\n_{L^{2p+2}}^{2p+2}\\
     			&\leq \frac{C}{\lambda_1}\n  u\n_2^2+C\n  u\n_2^{2p+2}\\
     			&\leq C_R,\qquad \forall t\in[0,T].
     		\end{split}
     	\end{equation}
     	Moreover, from (\ref{3-2-4}), (\ref{3-2-5}) and equations in (\ref{ode}), we also have
     	\begin{equation}\label{3-2-6}
     		\begin{split}
     			\n u^n_{tt}\n_{-2}&\leq \kappa\n u^n\n+\n u^n\n_2+\gamma\n(u^n,u^n_t)\n_{\h}^{2q}\n u^n_t\n+\n f(u^n)\n_{-2}\\
     			&\leq \frac{\kappa}{\sqrt{\lambda_1}}\n  u^n\n_2+\n u^n\n_2+\gamma\n(u^n,u^n_t)\n_{\h}^{2q}\n u^n_t\n+C\n f(u^n)\n\\
     			&\leq C_R,\qquad \forall t\in[0,T],
     		\end{split}     		
     	\end{equation}
     	and
     	\begin{equation}\label{3-2-7}
     		\begin{split}
     			&\qquad\int_{0}^t\left|\frac{d}{dt}\n(u^n,u^n_t)\n_{\h}^{2q}\right|d\tau\\
     			&=\int_0^tq\n(u^n,u^n_t)\n_{\h}^{2(q-1)}\left|\frac{d}{dt}(\n u^n\n_2^2+\n u_t^n\n^2)\right|d\tau\\
     			&\leq C_R\int_0^t\Big|(\kappa A^\frac{1}{2}u^n+\gamma\n(u^n,u^n_t)\n_{\h}^{2q}A^\frac{1}{2}u^n_t+f(u^n),u^n_t)\Big|d\tau\\
     			&\leq C_R\int_0^t\Big(\kappa\n u^n\n_2\n u^n_t\n+\gamma\n(u^n,u^n_t)\n_{\h}^{2q}\n u^n_t\n_1+\n f(u^n)\n\n u_t^n\n\Big)d\tau\\
     			&\leq tC_R+\gamma C_R\int_0^T\n(u^n,u^n_t)\n_{\h}^{2q}\n u^n_t\n_1d\tau\\
     			&\leq (1+t)C_R,\qquad \forall t\in[0,T].
     		\end{split}
     	\end{equation}
     	Estimate (\ref{3-2-7}) will be used later.
     	
     	Set $s=\frac{1}{q}$. By interpolation inequality (\ref{2-3}), we have
     	\begin{equation}\label{3-2-8}
     		\n u^n_t\n_s\leq \n u^n_t\n^\frac{1}{1+s}\n u^n_t\n_{1+s}^\frac{s}{1+s}.
     	\end{equation}
        Multiplying (\ref{ode}) by $\lambda_j^\frac{s}{2}T^{jn}$ for every $j=1,2,\cdots,n$ and sum up all the equations, we get
        \begin{align*}
        	&\frac{d}{dt}(u^n_t,A^\frac{s}{2}u^n)-\n u^n_t\n_s^2+\kappa\n u^n\n_{1+s}^2+\n u^n\n_{2+s}^2\\
        	&\qquad\qquad\qquad+\gamma\n(u^n,u^n_t)\n_\h^{2q}(A^\frac{1}{2}u^n_t,A^\frac{s}{2}u^n)+(f(u^n),A^\frac{s}{2}u^n)=0,
        \end{align*}
        Then, 
        \begin{equation}\label{3-2-9}
        	\begin{split}
        		&\quad\n u^n\n_{2+s}^2+\kappa\n u^n\n_{1+s}^2\\
        		&=\n u^n_t\n_s^2-\frac{d}{dt}(u^n_t,A^\frac{s}{2}u^n)-(f(u^n),A^\frac{s}{2}u^n)
        		-\gamma\n(u^n,u^n_t)\n_\h^{2q}(A^\frac{1}{2}u^n_t,A^\frac{s}{2}u^n)\\
        		&=\n u^n_t\n_s^2-\frac{d}{dt}(u^n_t,A^\frac{s}{2}u^n)-(f(u^n),A^\frac{s}{2}u^n)\\
        		&\quad-\frac{\gamma}{2}\frac{d}{dt}\left[\n(u^n,u^n_t)\n_\h^{2q}\n u^n\n_{1+s}^2\right]+\frac{\gamma}{2}\n u^n\n_{1+s}^2\frac{d}{dt}\n(u^n,u^n_t)\n_\h^{2q}\\
        		&=\n u^n_t\n_s^2-\frac{d}{dt}(u^n_t,A^\frac{s}{2}u^n)-\frac{\gamma}{2}\frac{d}{dt}\left[\n(u^n,u^n_t)\n_\h^{2q}\n u^n\n_{1+s}^2\right]+\varphi(t),
        	\end{split}
        \end{equation}
        where 
        \begin{align*}
        	\varphi(t):=-(f(u^n),A^\frac{s}{2}u^n)
        	+\frac{\gamma}{2}\n u^n\n_{1+s}^2\frac{d}{dt}\n(u^n,u^n_t)\n_\h^{2q}.
        \end{align*}
        Since $q\geq 1$, $0<s=\frac{1}{q}\leq1$, we have $V_2\hookrightarrow V_{1+s}\hookrightarrow V_{2s}$ and then $\varphi\in L^1(0,T)$ satisfying
        \begin{equation}\label{3-2-10}
        	\begin{split}
        		\int_{0}^{t}|\varphi(\tau)|d\tau&\leq\int_{0}^{t}|(f(u^n),A^\frac{s}{2}u^n)|d\tau+\frac{\gamma}{2}\int_{0}^{t}\n u^n\n_{1+s}^2\left|\frac{d}{dt}\n(u^n,u^n_t)\n_\h^{2q}\right|d\tau\\
        		&\leq \int_{0}^{t}\n f(u^n)\n\n u^n\n_{2}d\tau+\frac{\gamma C_R}{2}\int_{0}^{t}\left|\frac{d}{dt}\n(u^n,u^n_t)\n_\h^{2q}\right|d\tau\\
        		&\leq (1+t)C_R,\quad\forall t\in[0,T],
        	\end{split}
        \end{equation}
        where we have used estimates (\ref{3-2-4}), (\ref{3-2-5}) and (\ref{3-2-7}). Multiplying (\ref{ode}) by $\lambda_j^\frac{s}{2}T_t^{jn}$ for every $j=1,2,\cdots,n$ and sum up all the equations, we get 
     	\begin{equation}\label{3-2-11}
     		\frac{d}{dt}\Phi(u^n)+\gamma\n(u^n,u_t^n)\n_{\h}^{2q}\n u_t^n\n_{1+s}^2=-\frac{d}{dt}(f(u^n),A^\frac{s}{2}u^n)+(f'(u^n)u^n_t,A^\frac{s}{2}u^n),
     	\end{equation}
     	where
     	\[\Phi(u^n)=\frac{1}{2}\n u^n_t\n_s^2+\frac{\kappa}{2}\n u^n\n_{1+s}^2+\frac{1}{2}\n u^n\n_{2+s}^2\geq \frac{1}{2}\n(u^n,u^n_t)\n_{\h_s}^2.\]
     	By (\ref{3-1}), H\"{o}lder's inequality with $\frac{2}{n}+\frac{n-2s}{2n}+\frac{n-4+2s}{2n}=1$ and the Sobolev embeddings $V_s\hookrightarrow H^s\hookrightarrow L^\frac{2n}{n-2s}$, $V_{2-s}\hookrightarrow H^{2-s}\hookrightarrow L^\frac{2n}{n-4+2s}$, we have
     	\begin{equation}\label{3-2-12}
     		\begin{split}
     			|(f'(u^n)u^n_t,A^\frac{s}{2}u^n)|&\leq C|(u^n_t,A^\frac{s}{2}u^n)|+C|(|u^n|^pu^n_t,A^\frac{s}{2}u^n)|\\
     			&\leq C\n u^n_t\n\n u^n\n_{2s}+C\n |u^n|^p\n_{L^\frac{n}{2}}\n u^n_t\n_{L^\frac{2n}{n-2s}}\n A^\frac{s}{2}u^n\n_{L^\frac{2n}{n-4+2s}}\\
     			&\leq C\n u^n_t\n\n  u^n\n_2+C\n  u^n\n_2^p\n u^n_t\n_{s}\n u^n\n_{2+s}\\
     			&\leq C\n(u^n,u^n_t)\n_\h^2+C_R\n(u^n,u^n_t)\n_{\h_s}^2\\
     			&\leq C_R\Phi(u^n).
     		\end{split}
     	\end{equation}
     	Inserting (\ref{3-2-12}) into (\ref{3-2-11}) turns out
     	\begin{equation}\label{3-2-13}
     		\frac{d}{dt}\Phi(u^n)+\gamma\n(u^n,u_t^n)\n_{\h}^{2q}\n u_t^n\n_{1+s}^2\leq C_R\Phi(u^n)-\frac{d}{dt}(f(u^n),A^\frac{s}{2}u^n).
     	\end{equation}
     	When $0<t\leq T$, multiplying (\ref{3-2-13}) by $t^{1+s}$ and using (\ref{3-2-4}), (\ref{3-2-8}), (\ref{3-2-9}) as well as Young's inequality, we end up with
     	\begin{equation}\label{3-2-14}
     		\begin{split}
     			&\frac{d}{dt}\big(t^{1+s}\Phi(u^n)\big)+t^{1+s}\gamma\n(u^n,u_t^n)\n_{\h}^{2q}\n u_t^n\n_{1+s}^2-C_R\big(t^{1+s}\Phi(u^n)\big)\\
     			&\leq \frac{1+s}{2}\big(t^s\n u^n_t\n_s^2+t^s\kappa\n u^n\n_{1+s}^2+t^s\n u^n\n_{2+s}^2\big) -t^{1+s}\frac{d}{dt}(f(u^n),A^\frac{s}{2}u^n)\\
     			&\leq (1+s)t^s\n u^n_t\n_s^2+\frac{(1+s)t^s}{2}\varphi(t)-t^{1+s}\frac{d}{dt}(f(u^n),A^\frac{s}{2}u^n)\\
     			&\quad-\frac{(1+s)t^s}{2}\frac{d}{dt}(u^n_t,A^\frac{s}{2}u^n)-\frac{\gamma(1+s)t^s}{4}\frac{d}{dt}\left[\n(u^n,u^n_t)\n_\h^{2q}\n u^n\n_{1+s}^2\right]\\
     			&\leq (1+s)t^s\n u^n_t\n^\frac{2}{1+s}\n u^n_t\n_{1+s}^\frac{2s}{1+s}+t^s|\varphi(t)|+\psi(t)\\
     			&\leq \frac{st^{1+s}\gamma}{2}\n u^n_t\n^\frac{2}{s}\n u^n_t\n_{1+s}^2+\Big(\frac{\gamma}{2}\Big)^{-s}+t^s|\varphi(t)|+\psi(t)\\
     			&\leq \frac{t^{1+s}\gamma}{2}\n (u^n,u^n_t)\n_\h^{2q}\n u^n_t\n_{1+s}^2+t^s|\varphi(t)|+\psi(t)+C
     		\end{split}
     	\end{equation}
     	where
     	\begin{align*}
     		\psi(t):=&-t^{1+s}\frac{d}{dt}(f(u^n),A^\frac{s}{2}u^n)-\frac{(1+s)t^s}{2}\frac{d}{dt}(u^n_t,A^\frac{s}{2}u^n)\\
     		&-\frac{\gamma(1+s)t^s}{4}\frac{d}{dt}\left[\n(u^n,u^n_t)\n_\h^{2q}\n u^n\n_{1+s}^2\right].
     	\end{align*}
     	The inequality (\ref{3-2-14}) is equivalent to
     	\begin{equation}\label{3-2-15}
     		\begin{split}
     			\frac{d}{dt}\big(t^{1+s}\Phi(u^n)\big)&+\frac{t^{1+s}\gamma}{2}\n(u^n,u_t^n)\n_{\h}^{2q}\n u_t^n\n_{1+s}^2\\
     			&\leq C_R\big(t^{1+s}\Phi(u^n)\big)+t^s|\varphi(t)|+\psi(t)+C.
     		\end{split}
     	\end{equation}
     	Applying Gronwall's lemma to (\ref{3-2-15}) gives
     	\begin{equation}\label{3-2-16}
     		t^{1+s}\Phi(u^n(t))\leq \int_{0}^{t}e^{C_R(t-\tau)}\bigg(t^s|\varphi(\tau)|+\psi(\tau)+C\bigg)d\tau.
     	\end{equation} 
     	We need to estimate $\int_{0}^{t}e^{C_R(t-\tau)}\psi(\tau)d\tau$. Using integration by parts and (\ref{3-2-4})-(\ref{3-2-5}), we have for every $t\in[0,T]$,
     	\begin{align*}
     		&\quad-\int_{0}^{t}\tau^{1+s}e^{C_R(t-\tau)}\frac{d}{d\tau}(f(u^n),A^\frac{s}{2}u^n)d\tau\\
     		&=C_{R,T}-t^{1+s}(f(u^n(t)),A^\frac{s}{2}u^n(t))+\int_{0}^{t}(f(u^n),A^\frac{s}{2}u^n)((1+s)\tau^s-C_R\tau^{1+s})e^{C_R(t-\tau)}d\tau\\
     		&\leq T^{1+s}C_R+C_R\int_{0}^{t}(\tau^s+\tau^{1+s})e^{C_R(t-\tau)}d\tau\\
     		&\leq C_{R,T},
     	\end{align*}
     	\begin{align*}
     		&\quad-\frac{1+s}{2}\int_{0}^{t}\tau^se^{C_R(t-\tau)}\frac{d}{d\tau}(u^n_t,A^\frac{s}{2}u^n)d\tau\\
     		&=-\frac{(1+s)t^s}{2}(u^n_t(t),A^\frac{s}{2}u^n(t))+\frac{1+s}{2}\int_{0}^{t}(u^n_t,A^\frac{s}{2}u^n)(s\tau^{s-1}-C_R\tau^{s})e^{C_R(t-\tau)}d\tau\\
     		&\leq T^{s}C_R+C_R\int_{0}^{t}(\tau^{s-1}+\tau^{s})e^{C_R(t-\tau)}d\tau\\
     		&\leq C_{R,T},
     	\end{align*}
     	\begin{align*}
     		&\quad-\frac{\gamma(1+s)}{4}\int_{0}^{t}\tau^se^{C_R(t-\tau)}\frac{d}{d\tau}\left[\n(u^n,u^n_t)\n_\h^{2q}\n u^n\n_{1+s}^2\right]d\tau\\
     		&=-\frac{\gamma(1+s)t^s}{4}\n(u^n(t),u^n_t(t))\n_\h^{2q}\n u^n(t)\n_{1+s}^2\\
     		&\quad+\frac{\gamma(1+s)}{4}\int_{0}^{t}\n(u^n,u^n_t)\n_\h^{2q}\n u^n\n_{1+s}^2(s\tau^{s-1}-C_R\tau^{s})e^{C_R(t-\tau)}d\tau\\
     		&\leq T^{s}C_R+C_R\int_{0}^{t}(\tau^{s-1}+\tau^{s})e^{C_R(t-\tau)}d\tau\\
     		&\leq C_{R,T}.
     	\end{align*}
     	Inserting these estimates into (\ref{3-2-16}) and using (\ref{3-2-10}), we end up with 
     	\begin{equation}
     		\begin{split}
     			t^{1+s}\Phi(u^n(t))\leq C_{R,T},\qquad \forall 0<t\leq T.
     		\end{split}
     	\end{equation}
     	By the definition of $\Phi$, we get
     	\begin{equation}\label{3-2-18}
     		\n(u^n(t),u^n_t(t))\n_{\h_s}^2\leq \frac{C_{R,T}}{t^{1+s}},\qquad\forall 0<t\leq T.
     	\end{equation}
        Moreover, for any $0<t\leq t'\leq T$, integrate (\ref{3-2-13}) over $[t,t']$ and use (\ref{3-2-4})-(\ref{3-2-5}), we have
        \begin{equation}\label{3-2-19}
        	\int_{t}^{t'}\n(u^n(\tau),u_t^n(\tau))\n_{\h}^{2q}\n u_t^n(\tau)\n_{1+s}^2d\tau\leq \frac{C_{R,T}}{t^{1+s}}.
        \end{equation}
     	
     	\textbf{Existence of weak solutions.} From (\ref{3-2-4})-(\ref{3-2-6})  and (\ref{3-2-18})-(\ref{3-2-19}), we can extract a subsequence (still denoted by itself) such that
     	\begin{equation}\label{lim1}
     		\begin{split}
     			&(u^n,u^n_t)\rightarrow (u,u_t)\ \ \mathrm{weakly^*}\ \ \mathrm{in}\ \ L^\infty(0,T;\h),\\
     			&(u^n,u^n_t)\rightarrow (u,u_t)\ \ \mathrm{weakly^*}\ \ \mathrm{in}\ \ L^\infty(a,T;\h_s)\ \ \mathrm{for}\ \ \mathrm{any}\ \ a>0,\\
     			&\n(u^n,u^n_t)\n_\h^qA^\frac{1}{4}u^n_t\rightarrow\eta\ \ \mathrm{weakly}\ \ \mathrm{in}\ \ L^2(0,T;L^2)\ \ \mathrm{for}\ \ \mathrm{some}\ \ \eta\in L^2(0,T;L^2),\\
     			&\n(u^n,u^n_t)\n_\h^qA^\frac{1+s}{4}u^n_t\rightarrow{\eta^a}\ \ \mathrm{weakly}\ \ \mathrm{in}\ \ L^2(a,T;L^2)\ \ \mathrm{for}\ \ \mathrm{some}\ \ {\eta^a}\in L^2(a,T;L^2),\\
     			&u^n_{tt}\rightarrow u_{tt}\ \ \mathrm{weakly^*}\ \ \mathrm{in}\ \ L^\infty(0,T;V_{-2}).
     		\end{split}
     	\end{equation}
     	Since $\h_s\hookrightarrow\hookrightarrow\h\hookrightarrow\hookrightarrow \h_{-2}$, applying Aubin-Lions Lemma \cite{Lions} yields
     	\begin{equation}\label{lim2}
     		\begin{split}
     			&(u^n,u^n_t)\rightarrow (u,u_t)\ \ \mathrm{strongly}\ \ \mathrm{in}\ \ C([0,T];\h_{-2}),\\
     			&(u^n,u^n_t)\rightarrow (u,u_t)\ \ \mathrm{strongly}\ \ \mathrm{in}\ \ C([a,T];\h)\ \ \mathrm{for}\ \ \mathrm{any}\ \ a>0.
     		\end{split}
     	\end{equation}
     	By the continuity of $f$ and (\ref{3-2-5}), we also have 
     	\begin{equation}\label{lim3}
     		\begin{split}
     			&u^n\rightarrow u\ \ a.e.\ \ \mathrm{in}\ \ \Omega\times[0.T],\\
     			&f(u^n)\rightarrow f(u)\ \ \mathrm{a.e.}\ \ \mathrm{in}\ \ \Omega\times[0,T],\\
     			&f(u^n)\rightarrow f(u)\ \ \mathrm{weakly}\ \ \mathrm{in}\ \ L^2(0,T;L^2).
     		\end{split}
     	\end{equation}
     	In particular, (\ref{lim2}) implies that
     	\begin{equation}\label{lim4}
     		\n(u^n,u^n_t)\n_\h^{2q}\rightarrow\n(u,u_t)\n_\h^{2q}\ \ \mathrm{a.e.}\ \ t\in[0,T].
     	\end{equation}
     	Combining (\ref{lim1}) and (\ref{lim4}), we can prove that
     	\begin{equation}\label{lim5}
     		\n(u^n,u^n_t)\n_\h^{2q}A^\frac{1}{2} u^n_t\rightarrow\n(u,u_t)\n_\h^{2q}A^\frac{1}{2} u_t\ \ \mathrm{weakly^*}\ \ \mathrm{in}\ \ L^\infty(0,T;V_{-2}).
     	\end{equation}
     	Indeed, $\forall \phi\in L^1(0,T;V_2)$, by Lebesgue dominated convergence theorem,
     	\begin{align*}
     		&\quad\bigg|\int_{0}^{T}(\n(u^n,u^n_t)\n_\h^{2q}A^\frac{1}{2} u^n_t-\n(u,u_t)\n_\h^{2q}A^\frac{1}{2} u_t,\phi)dt\bigg|\\
     		&\leq \bigg|\int_{0}^{T}\n(u^n,u^n_t)\n_\h^{2q}(A^\frac{1}{2} u^n_t-A^\frac{1}{2} u_t,\phi)dt\bigg|+\bigg|\int_{0}^{T}\big[\n(u^n,u^n_t)\n_\h^{2q}-\n(u,u_t)\n_\h^{2q}\big](A^\frac{1}{2} u_t,\phi)dt\bigg|\\
     		&\leq C_R\bigg|\int_{0}^{T}( u^n_t- u_t,A^\frac{1}{2}\phi)dt\bigg|+C_{R}\int_{0}^{T}\big|\n(u^n,u^n_t)\n_\h^{2q}-\n(u,u_t)\n_\h^{2q}\big|\n\phi\n_2 dt\\
     		&\rightarrow0,\quad as\quad n\rightarrow\infty.
     	\end{align*}
     	Now, with the aid of the limits (\ref{lim1})-(\ref{lim5}), by passing to the limit $n\rightarrow\infty$ in (\ref{ode}), we obtain that the limit function $(u,u_t)\in L^\infty(0,T;\h)$ satisfies $$(u(0),u_t(0))\overset{\h_{-2}}{=}\lim_{n\rightarrow\infty}(u^n_0,u^n_1)=(u_0,u_1)$$ and Eq. (\ref{eq}) in the sense that
     	\[u_{tt}+\kappa A^\frac{1}{2}u+Au+\gamma\n(u,u_t)\n_{\h}^{2q}A^\frac{1}{2}u_t+f(u)=0\ \ in\ \ L^\infty(0,T;V_{-2}).\]
     	Thus, $u(t), t\in[0,T]$ is a weak solution to problem (\ref{eq}) and $(u,u_t)\in L^\infty(a,T;\h_s),\forall 0<a<T$ by (\ref{lim1}). 
     	\\
     	
     	\textbf{(ii)} We will show that $\n(u,u_t)\n_\h^qA^\frac{1}{4}u_t=\eta$ and $\n(u,u_t)\n_\h^qA^\frac{1+s}{4}u_t=\eta^a$ with $\eta,\eta^a$ coming from (\ref{lim1}). Indeed, for $\forall\phi\in L^2(0,T;V_1)$, by (\ref{lim4}) and Lebesgue dominated convergence theorem,
     	\begin{align*}
     		&\quad\bigg|\int_{0}^{T}(\n(u^n,u^n_t)\n_\h^{q}A^\frac{1}{4} u^n_t-\n(u,u_t)\n_\h^{q}A^\frac{1}{4} u_t,\phi)dt\bigg|\\
     		&\leq \bigg|\int_{0}^{T}\n(u^n,u^n_t)\n_\h^{q}(A^\frac{1}{4} u^n_t-A^\frac{1}{4} u_t,\phi)dt\bigg|+\bigg|\int_{0}^{T}\big[\n(u^n,u^n_t)\n_\h^{q}-\n(u,u_t)\n_\h^{q}\big](A^\frac{1}{4} u_t,\phi)dt\bigg|\\
     		&\leq C_R\bigg|\int_{0}^{T}( u^n_t- u_t,A^\frac{1}{4}\phi)dt\bigg|+C_{R}\int_{0}^{T}\big|\n(u^n,u^n_t)\n_\h^{q}-\n(u,u_t)\n_\h^{q}\big|\n\phi\n_1 dt\\
     		&\rightarrow0,\quad as\quad n\rightarrow\infty,
     	\end{align*}
        which ipmlies that $\n(u^n,u^n_t)\n_\h^qA^\frac{1}{4}u^n_t\rightarrow\n(u,u_t)\n_\h^qA^\frac{1}{4}u_t$ weakly in $L^2(0,T,V_{-1})$. By the uniqueness of weak limit, we obtain that $\n(u,u_t)\n_\h^qA^\frac{1}{4}u_t=\eta\in L^2(0,T;L^2)$. Similarly, for any $a>0$, $\forall\phi\in L^2(a,T;V_{1+s})$, by (\ref{lim4}) and Lebesgue dominated convergence theorem,
        \begin{align*}
        	&\quad\bigg|\int_{a}^{T}(\n(u^n,u^n_t)\n_\h^{q}A^\frac{1+s}{4} u^n_t-\n(u,u_t)\n_\h^{q}A^\frac{1+s}{4} u_t,\phi)dt\bigg|\\
        	&\leq \bigg|\int_{a}^{T}\n(u^n,u^n_t)\n_\h^{q}(A^\frac{1+s}{4} u^n_t-A^\frac{1+s}{4} u_t,\phi)dt\bigg|+\bigg|\int_{a}^{T}\big[\n(u^n,u^n_t)\n_\h^{q}-\n(u,u_t)\n_\h^{q}\big](A^\frac{1+s}{4} u_t,\phi)dt\bigg|\\
        	&\leq C_R\bigg|\int_{a}^{T}( u^n_t- u_t,A^\frac{1+s}{4}\phi)dt\bigg|+C_{R}\int_{a}^{T}\big|\n(u^n,u^n_t)\n_\h^{q}-\n(u,u_t)\n_\h^{q}\big|\n\phi\n_{1+s} dt\\
        	&\rightarrow0,\quad as\quad n\rightarrow\infty,
        \end{align*}
        which ipmlies that $\n(u^n,u^n_t)\n_\h^qA^\frac{1+s}{4}u^n_t\rightarrow\n(u,u_t)\n_\h^qA^\frac{1+s}{4}u_t$ weakly in $L^2(a,T,V_{-1-s})$. By the uniqueness of weak limit, we obtain that $\n(u,u_t)\n_\h^qA^\frac{1+s}{4}u_t=\eta^a\in L^2(a,T;L^2)$.
        
        Then, due to the lower semicontinuity of the weak$^*$ limit, we can deduce from (\ref{3-2-4}), (\ref{3-2-18}), (\ref{3-2-19}) and (\ref{lim1}) that
        \begin{equation*}
        	\n(u(t),u_t(t))\n_{\h}^2+\int_{0}^{t}\n(u(\tau),u_t(\tau))\n_{\h}^{2q}\n u_t(\tau)\n_1^2d\tau\leq C_R,\ \forall t\in[0,T],
        \end{equation*}
        \begin{equation*}
        	\n(u(t),u_t(t))\n_{\h_s}^2+\int_{t}^{t'}\n(u(\tau),u_t(\tau))\n_{\h}^{2q}\n u_t(\tau)\n_{1+s}^2d\tau
        	\leq \frac{C_{R,T}}{t^{1+s}},\ \forall 0<t\leq t'\leq T,
        \end{equation*}
     	\\
     		
     	\textbf{(iii)} From (\ref{3-3}) and (\ref{3-7}), we have
     	\begin{align*}
     		u_t\in L^2(0,T; L^2),\qquad
     		f(u)\in L^2(0,T;L^2).
     	\end{align*}
     	Then, approximating $u$ by smooth functions and arguing in a standard way, we see that for every $0\leq \tau\leq t\leq T$, 
     	\begin{equation}\label{3-2-21}
     		\begin{split}
     			\int_{\Omega}F(u(t))dx-\int_{\Omega}F(u(\tau))dx=\int_{\tau}^{t}(f(u),u_t)ds.
     		\end{split}
     	\end{equation}

     	We are now ready to prove the energy equality of $u$. To this end, we take $u^m(t)=P_mu(t)$, where $P_m$ is the orthoprojector to $\mathrm{Span}\{e_1,\cdots,e_m\}$ with $\{e_j\}_{j=1}^m$ the first $m$ eigenvalues of the operator $A$. Then, $u^m$ solves
     	\begin{equation}\label{3-2-22}
     		u^m_{tt}+\kappa A^\frac{1}{2}u^m+Au^m+\gamma\n(u,u_t)\n_{\h}^{2q}A^\frac{1}{2}u^m_t+P_mf(u)=0
     	\end{equation}
     	For every $0\leq\tau\leq t\leq T$, take the multiplier $u_t^m$ in (\ref{3-2-22}) and integrate over $[\tau,t]$, we get
     	\begin{equation}\label{3-2-23}
     		\begin{split}
     			&\frac{1}{2}\n (u^m(t),u^m_t(t))\n_\h^2+\frac{\kappa}{2}\n u^m(t)\n_1^2-\frac{1}{2}\n (u^m(\tau),u^m_t(\tau))\n_\h^2-\frac{\kappa}{2}\n u^m(\tau)\n_1^2\\
     			&\qquad +\gamma\int_{\tau}^{t}\n(u,u_t)\n_\h^{2q}\n A^\frac{1}{4}u_t^m\n^2ds+\int_{\tau}^{t}(P_mf(u),u_t^m)ds.
     		\end{split}
     	\end{equation}
     	Since $(u,u_t)\in L^\infty(\tau,t;\h),\ f(u)\in L^2(\tau,t;L^2),\ \n(u,u_t)\n_\h^qA^\frac{1}{4}u_t\in L^2(0,T;L^2)$, we have
     	\begin{align*}
     		&(u^m(t),u^m_t(t))\rightarrow(u(t),u_t(t)),\ (u^m(\tau),u^m_t(\tau))\rightarrow(u(\tau),u_t(\tau))\ \ \mathrm{in}\ \ \h;\\
     		&u^m(t)\rightarrow u(t),\ u^m(\tau)\rightarrow u(\tau)\ \ \mathrm{in}\ \ V_1;\\
     		&u_t^m\rightarrow u_t,\ \n(u,u_t)\n_\h^qA^\frac{1}{4}u^m_t\rightarrow\n(u,u_t)\n_\h^qA^\frac{1}{4}u_t\ \ \mathrm{in}\ \ L^2(\tau,t;L^2).
     	\end{align*}
     	Now, passing to the limit $m\rightarrow\infty$ in (\ref{3-2-23}) and with the help of (\ref{3-2-21}), we end up with
     	\begin{equation*}
     		E(t)-E(\tau)+\gamma\int_{\tau}^{t}\n(u,u_t)\n_\h^{2q}\n u_t\n_1^2=0,
     	\end{equation*}
     	where 
     	\begin{equation*}
     		E(t)=\frac{1}{2}\n(u(t),u_t(t))\n_\h^2+\frac{\kappa}{2}\n u(t)\n_1^2+\int_{\Omega}F(u(t))dx.
     	\end{equation*}
     	Thus, (\ref{energy}) holds for $u$. From the energy equality, the function $t\mapsto\n (u(t),u_t(t))\n_\h$ is continuous. Moreover, since $(u,u_t)\in L^\infty(0,T;\h)\cap C([0,T],\h_{-2})$ and $\h\hookrightarrow\h_{-2}$ is reflexive, we have $(u,u_t)\in C_w([0,T];\h)$. Then, the uniform convexity of $\h$ gives that $(u,u_t)\in C([0,T];\h)$.
    \end{proof}
     	
    We prove the following lemma in order to verify the uniqueness of weak solutions.
    	
    \begin{lemma}\label{lemma1}
    	For any fixed $R>0$ and for any compact set $K$ with 
    	\[K\subset\Big\{\xi\in\h\ |\ \frac{R}{2}\leq\n\xi\n_\h\leq R\Big\},\]
    	there exists a constant $k=k(R,T,K)>0$ such that the inequality
    	\begin{equation}\label{3-2-30}
    		\n u(t),u_t(t)\n_\h\geq k,\quad \forall t\in[0,T]
    	\end{equation}
    	holds uniformly for any weak solution $u(t), t\in[0,T]$ with its initial data $(u_0,u_1)\in K$.
    \end{lemma}
    \begin{proof}
    	Firstly, we prove a useful result for every weak solution. Let $r\geq0$ be arbitrary and $u(t), t\in[0,T]$ be a weak solution of problem \ref{eq} with its initial data $\n(u(0),u_t(0))\n_\h\leq r$. Multiplying (\ref{eq}) by $u_t$ and integrating over $x\in\Omega$ (which is justified as in Theorem \ref{well-posedness} (iii)), we get
    	\begin{equation}\label{3-2-36}
    		\frac{d}{dt}\e(u)+2\gamma\n (u,u_t)\n_\h^{2q}\n u_t\n_1^2+2(f(u),u_t)=0,
    	\end{equation}
    	where $\e(u)=\n u_t\n^2+\kappa\n u\n_{1}^2+\n u\n_{2}^2$ satisfying $\n (u,u_t)\n_\h^{2}\leq \e(u)\leq C_0\n (u,u_t)\n_\h^{2}$ with the constant $C_0=1+\kappa/(\lambda_1^{1/2})>0$. By (\ref{3-3}), (\ref{3-7}),  H\"{o}lder's inequality with $\frac{2}{n}+\frac{n-4}{2n}+\frac{1}{2}=1$ and the Sobolev embedding $V_2\hookrightarrow L^{\frac{2n}{n-4}}\hookrightarrow L^{\frac{np}{2}}$, we have
    	\begin{equation}\label{3-2-35}
    		\begin{split}
    			\Big|2(f(u),u_t)\Big|
    			&\leq C((1+|u|^p)|u|,|u_t|)\\
    			&\leq C(1+\n |u|^p\n_{L^\frac{n}{2}})\n u\n_{L^\frac{2n}{n-4}}\n u_t\n\\
    			&\leq C(1+\n u\n^p_2)\n u\n_{2}\n u_t\n\\
    			&\leq C_r\n (u,u_t)\n_\h^{2}\\
    			&\leq C_r\e(u).
    		\end{split}
    	\end{equation}
    	Inserting (\ref{3-2-35}) into (\ref{3-2-36}), we obtain
    	\begin{equation}\label{3-2-50}
    		\frac{d}{dt}\e(u)+2\gamma\n (u,u_t)\n_\h^{2q}\n u_t\n_1^2\leq C_r\e(u).
    	\end{equation}
    	Applying Gronwall's lemma to (\ref{3-2-50}) and using the definition of $\e(u)$, we see that for every $0\leq \tau\leq t\leq T$,
    	\begin{equation}\label{3-2-51}
    		\n(u(t),u_t(t))\n^2_\h\leq \e(u(t))\leq \e(u(\tau))e^{C_r(t-\tau)}\leq C_{r,T}\n(u(\tau),u_t(\tau))\n_\h^2. 
    	\end{equation}
    	
    	Now, we prove the lemma by contradiction. Suppose that (\ref{3-2-30}) does not hold true. Then for $n\geq 1$, there exist $t_n\in[0,T]$ and weak solutions ${u^n}$ with initial data $({u^n}(0),u^n_t(0))\in K$, such that
    	\begin{equation}\label{3-2-53}
    		\n(u^n(t_n),u_t^n(t_n))\n_\h<\frac{1}{n}.
    	\end{equation}
    	By (\ref{3-2-51})and (\ref{3-2-53}), we have
    	\begin{equation}\label{3-2-54}
    		\n(u^n(T),u_t^n(T))\n_\h\leq C_{R,T}\n(u^n(t_n),u_t^n(t_n))\n_\h<\frac{C_{R,T}}{n},\quad n\geq1.
    	\end{equation}
    	Since $K$ is compact, we may assume that $(u^n(0),u^n_t(0))\rightarrow(u_0,u_1)\in K$ in $\h$. Arguing exactly as in the proof of Theorem \ref{well-posedness} (i), we deduce that $\{u^n\}_{n\geq1}$ (subsequence if necessary) converges to some weak solution $u$ satisfying
    	\begin{equation}\label{3-2-55}
    		\begin{split}
    			&(u^n,u^n_t,u^n_{tt})\rightarrow ({u},{u}_t,u_{tt})\ \ \mathrm{weakly^*}\ \ \mathrm{in}\ \ L^\infty(0,T;V_2\times L^2\times V_{-2}),\\
    			&(u(0),u_t(0))=(u_0,u_1)\in K,
    		\end{split}
    	\end{equation}
    	By the weak$^*$ convergence in (\ref{3-2-55}), we have the weak convergence $(u^n(t),u^n_t(t))\rightharpoonup(u(t),u_t(t))$ in $\h$ for every $t\in[0,T]$. Consequently, due to (\ref{3-2-54}) and the lower semicontinuity of the weak limit, we have
    	\begin{equation}\label{3-2-56}
    		\n(u(T),u_t(T))\n_\h\leq \liminf_{n\rightarrow\infty}\n(u^n(T),u^n_t(T))\n_\h\leq \liminf_{n\rightarrow\infty}\frac{C_{R,T}}{n}=0.
    	\end{equation}
    	
    	On the other hand, since $\frac{R}{2}\leq \n(u(0),u_t(0))\n_\h\leq R$ and $t\mapsto\n(u(t),u_t(t))\n_\h$ is continuous due to the energy equality (\ref{energy}), there exists some $0<a\leq T$ such that
     	\begin{equation}\label{3-2-31}
     		\n(u(t),u_t(t))\n_\h\geq \frac{R}{4},\quad\forall t\in[0,a].
     	\end{equation}
     	Since estimate (\ref{3-8}) holds for $u$, we have
     	\begin{equation}\label{3-2-33}
     		\n(u(t),u_t(t))\n^2_{\h_s}+\int_{a}^{t}\n(u,u_t)\n_\h^{2q}\n u_t\n_{1+s}^2d\tau\leq a^{-(1+s)}C_{R,T},\quad \forall t\in[a,T],
     	\end{equation}
     	where $s=\frac{1}{q}$. 
%     	Similarly to (\ref{3-2-11}) and (\ref{3-2-12}), multiplying (\ref{eq}) by $A^\frac{s}{2}u_t$ and integrating over $\Omega\times[a,t]$ (which can be justified similarly as in (ii)), with the help of estimates (\ref{3-2-55}) and (\ref{3-2-32}), we get
%     	\begin{equation}\label{3-2-32}
%     		\begin{split}
%     			\gamma\int_{a}^{t}\n(u,u_t)\n_\h^{2q}\n u_t\n_{1+s}^2d\tau&=\Phi(u(a))-\Phi(u(t))+(f(u(a)),A^\frac{s}{2}u(a))\\
%     			&\quad-(f(u(t)),A^\frac{s}{2}u(t))+\int_{a}^{t}(f'(u)u_t,A^\frac{s}{2}u)d\tau\\
%     			&\leq \Phi(u(a))+\Phi(u(t))+\n f(u(a))\n\n u(a)\n_{2s}\\
%     			&\quad+\n f(u(t))\n\n u(t)\n_{2s}+C_R\int_{a}^{t}\Phi(u(\tau))d\tau\\
%     			&\leq a^{-(1+s)}C_{R,T}, \quad\forall t\in[a,T].
%     		\end{split}
%     	\end{equation}
%     	where $\Phi(u)=\frac{1}{2}\n u_t\n_s^2+\frac{\kappa}{2}\n u\n_{1+s}^2+\frac{1}{2}\n u\n_{2+s}^2\leq C\n(u,u_t)\n_{\h_s}^2$ for some $C>0$. 
        By interpolation inequality (\ref{2-3}), we have
     	\[\n u_t\n_1\leq \n u_t\n^\frac{s}{1+s}\n u_t\n_{1+s}^\frac{1}{1+s},\]
     	and then 
     	\begin{equation}\label{3-2-34}
     		\begin{split}
     			\n (u,u_t)\n_\h^{2q}\n u_t\n_1^2&\leq \n (u,u_t)\n_\h^{2q}\n u_t\n^\frac{2s}{1+s}\n u_t\n_{1+s}^\frac{2}{1+s}\\
     			&\leq \n (u,u_t)\n_\h^{\frac{2}{s}+\frac{2s}{1+s}}\n u_t\n_{1+s}^\frac{2}{1+s}\\
     			&=\n (u,u_t)\n_\h^{2}\Big(\n (u,u_t)\n_\h^{\frac{2}{s}}\n u_t\n_{1+s}^2\Big)^\frac{1}{1+s}.
     		\end{split}
     	\end{equation}
     	Inserting (\ref{3-2-35}) and (\ref{3-2-34}) into (\ref{3-2-36}), we obtain
     	\begin{equation}\label{3-2-37}
     		\frac{d}{dt}\e(u)+\bigg[2\gamma\Big(\n (u,u_t)\n_\h^{2q}\n u_t\n_{1+s}^2\Big)^\frac{q}{1+q}+C_R\bigg]\e(u)\geq 0.
     	\end{equation}
     	 Integrating (\ref{3-2-37}) over $[a,t]$ yields
     	 \begin{equation}\label{3-2-38}
     	 	\e(u(t))\geq \e(u(a))e^{-\int_{a}^{t}h(\tau)d\tau},\quad\forall t\in[a,T],
     	 \end{equation}
     	 where $h(t)=2\gamma\Big(\n (u,u_t)\n_\h^{2q}\n u_t\n_{1+s}^2\Big)^\frac{q}{1+q}+C_R$. By (\ref{3-2-33}) and H\"{o}lder's inequality, we have
     	 \begin{equation}\label{3-2-39}
     	 	\begin{split}
     	 		\int_{a}^{t}h(\tau)d\tau&=2\gamma\int_{a}^{t}\Big(\n (u,u_t)\n_\h^{2q}\n u_t\n_{1+s}^2\Big)^\frac{q}{1+q}d\tau+C_R(t-a)\\
     	 		&\leq 2\gamma\Big(\int_{a}^{t}\n (u,u_t)\n_\h^{2q}\n u_t\n_{1+s}^2d\tau\Big)^\frac{q}{1+q}(t-a)^\frac{1}{1+q}+C_R(t-a)\\
     	 		&\leq \frac{C_{R,T}}{a}(t-a)^\frac{1}{1+q}+C_R(t-a).
     	 	\end{split}
     	 \end{equation}
     	Inserting (\ref{3-2-31}) and (\ref{3-2-39}) into (\ref{3-2-38}), we obtain
     	\begin{equation*}
     		\begin{split}
     			\n (u(t),u_t(t))\n_\h^2&\geq\frac{1}{C_0}\e(u(t))\\
     			&\geq\frac{1}{C_0}\e(u(a))\exp\Big\{-\frac{C_{R,T}}{a}(t-a)^\frac{1}{1+q}-C_R(t-a)\Big\}\\
     			&\geq\frac{1}{C_0}\n (u(a),u_t(a))\n_\h^2\exp\Big\{-\frac{C_{R,T}}{a}(t-a)^\frac{1}{1+q}-C_R(t-a)\Big\}\\
     			&\geq\frac{R^2}{16C_0}\exp\Big\{-\frac{T^\frac{1}{1+q}C_{R,T}}{a}-TC_R\Big\},\quad\forall t\in[a,T].
     		\end{split}
     	\end{equation*}
     	In particular,
     	\begin{equation}\label{3-2-40}
     		\n u(T),u_t(T)\n_\h\geq\frac{R}{4\sqrt{C_0}}\exp\Big\{-\frac{T^\frac{1}{1+q}C_{R,T}}{2a}-\frac{TC_R}{2}\Big\}>0,
     	\end{equation}
     	which contradicts with (\ref{3-2-56}). Therefore, (\ref{3-2-30}) holds true.
    \end{proof}

    \begin{thm}\label{uniqueness}
    	Let the assumptions of Theorem \ref{well-posedness} hold. Then the weak solution is unique. Specifically, for any $R>0$ and any compact set $K$ with 
    	\[K\subset\Big\{\xi\in\h\ |\ \frac{R}{2}\leq\n\xi\n_\h\leq R\Big\},\]
    	if $\xi_u(t)=(u(t),u_t(t))$ and $\xi_v(t)=(v(t),v_t(t))$ are weak solutions corresponding to the initial data $\xi_u(0)\in K$, $\xi_v(0)\in K$, respectively, then
    	\begin{equation}\label{unique}
    		\n\xi_{u}(t)-\xi_{v}(t)\n_{\h}\leq C\n\xi_{u}(0)-\xi_{v}(0)\n_{\h},\quad \forall t\in[0,T],
    	\end{equation}
    	where the constant $C=C(R,T,K)>0$. Moreover, if the initial data satisfies $\n\xi_u(0)\n_\h=0$, then $u\equiv0$ is the only solution.
    \end{thm}
    \begin{proof}
     	Firstly, if $\n\xi_u(0)\n_\h=0$, we deduce from (\ref{3-2-51}) that
     	\begin{equation}
     		\n(u(t),u_t(t))\n^2_\h\leq C_T\n\xi_u(0)\n_\h^2=0.\quad\forall t\in[0,T].
     	\end{equation}
     	Thus $(u,u_t)\equiv0$ is the only solution.
     	
     	For fixed $R>0$, let $K\subset\{\xi\in\h\ |\ \frac{R}{2}\leq\n\xi\n_\h\leq R\}$ be an arbitrary compact set and $\xi_u(t)=(u(t),u_t(t))$ , $\xi_v(t)=(v(t),v_t(t))$ be weak solutions corresponding to the initial data $\xi_u(0)\in K$, $\xi_v(0)\in K$, respectively. By estimate (\ref{3-7}) and Lemma \ref{lemma1},
     	\begin{equation}\label{3-6-1}
     		k\leq\n(u(t),u_t(t))\n_\h\leq C_R,\quad k\leq\n(v(t),v_t(t))\n_\h\leq C_R,\quad\forall t\in[0,T],
     	\end{equation}
     	where $k=k(R,T,K)>0$. Then, from the energy equality (\ref{energy}), we have
     	\begin{equation}\label{3-6-2}
     		\int_{0}^{t}\n u_t\n_1^2d\tau\leq \frac{C_R}{\gamma k^{2q}},\quad \int_{0}^{t}\n v_t\n_1^2d\tau\leq \frac{C_R}{\gamma k^{2q}},\quad\forall t\in[0,T].
     	\end{equation}
     	The function $w(t)=u(t)-v(t)$ solves the following equation
     	\begin{equation}\label{eqw}
     		\begin{cases}
     			w_{tt}+\kappa A^\frac{1}{2}w+Aw+\gamma\n(u,u_t)\n_\h^{2q}A^\frac{1}{2}w_t\\
     			\qquad\qquad+\gamma\Big[\n(u,u_t)\n_\h^{2q}-\n(v,v_t)\n_\h^{2q}\Big]A^\frac{1}{2}v_t+f(u)-f(v)=0,\\
     			(w(0),w_t(0))=\xi_u(0)-\xi_v(0).
     		\end{cases}
     	\end{equation}
     	Multiplying (\ref{eqw}) by $w_t$, we get
     	\begin{equation}\label{3-6-3}
     		\begin{split}
     			\frac{d}{dt}\e(w)+2\gamma\n (u,u_t)\n_\h^{2q}\n w_t\n_1^2=I_1+I_2,
     		\end{split}
     	\end{equation}
     	where $\e(w)=\n w_t\n^2+\kappa\n w\n_{1}^2+\n w\n_{2}^2$ satisfying $\n (w,w_t)\n_\h^{2}\leq \e(w)\leq C_0\n (w,w_t)\n_\h^{2}$ with the constant $C_0=1+\kappa/(\lambda_1^{1/2})>0$ and
     	\begin{align*}
     		&I_1=-2\gamma\Big[\n(u,u_t)\n_\h^{2q}-\n(v,v_t)\n_\h^{2q}\Big](A^\frac{1}{4}v_t,A^\frac{1}{4}w_t),\\
     		&I_2=-2(f(u)-f(v),w_t).
     	\end{align*}
     	From (\ref{3-6-1}), using mean value theorem and Young's inequality, we have
     	\begin{equation}\label{3-6-4}
     		\begin{split}
     			|I_1|&\leq 2q\gamma\Big[\n(u,u_t)\n_\h^{2(q-1)}+\n(v,v_t)\n_\h^{2(q-1)}\Big]\Big(\n(u,u_t)\n_\h^{2}-\n(v,v_t)\n_\h^{2}\Big)\n v_t\n_1\n w_t\n_1\\
     			&\leq C_R\Big(\n(u,u_t)\n_\h+\n(v,v_t)\n_\h\Big)\Big(\n(u,u_t)\n_\h-\n(v,v_t)\n_\h\Big)\n v_t\n_1\n w_t\n_1\\
     			&\leq C_R\n(w,w_t)\n_\h\n v_t\n_1\n w_t\n_1\\
     			&\leq \gamma k^{2q}\n w_t\n_1^2+C_{k,R}\n(w,w_t)\n_\h^2\n v_t\n_1^2\\
     			&\leq \gamma k^{2q}\n w_t\n_1^2+C_{k,R}\n v_t\n_1^2\e(w).
     		\end{split}
     	\end{equation}
     	By (\ref{3-1}), mean value theorem,  H\"{o}lder's inequality with $\frac{2}{n}+\frac{n-4}{2n}+\frac{1}{2}=1$ and the Sobolev embedding $V_2\hookrightarrow L^{\frac{2n}{n-4}}\hookrightarrow L^{\frac{np}{2}}$, we have
     	\begin{equation}\label{3-6-5}
     		\begin{split}
     			|I_2|
     			&\leq C((1+|u|^p+|v|^p)|w|,|w_t|)\\
     			&\leq C(1+\n |u|^p\n_{L^\frac{n}{2}}+\n |v|^p\n_{L^\frac{n}{2}})\n w\n_{L^\frac{2n}{n-4}}\n w_t\n\\
     			&\leq C(1+\n u\n^p_2+\n v\n^p_2)\n w\n_{2}\n w_t\n\\
     			&\leq C_R\n (w,w_t)\n_\h^{2}\\
     			&\leq C_R\e(w).
     		\end{split}
     	\end{equation}
     	Inserting (\ref{3-6-4}) and (\ref{3-6-5}) into (\ref{3-6-3}) and using (\ref{3-6-1}), we get
     	\begin{equation*}
     		\frac{d}{dt}\e(w)+\gamma k^{2q}\n w_t\n_1^2\leq C_{k,R}(1+\n v_t\n_1^2)\e(w).
     	\end{equation*}
     	Applying Gronwall's inequality and using (\ref{3-6-2}), we obtain
     	\[\e(w(t))\leq \e(w(0))e^{\int_{0}^{t}C_{k,R}(1+\n v_t\n_1^2)ds}\leq C_{k,R,T}\e(w(0)),\quad \forall t\in[0,T].\]
     	Then, we conclude that
     	\[\n(w(t),w_t(t))\n_\h^2\leq \e(w(t))\leq C\e(w(0))\leq C\n(w(0),w_t(0))\n_\h^2,\quad\forall t\in[0,T],\]
     	where $C=C(R,T,K)$, which implies that (\ref{unique}) holds.
     	
     	In particular, taking $\xi_u(0)=\xi_v(0)$ in (\ref{unique}), we get the uniqueness for initial data with nonzero $\h$-norm. 
     \end{proof}
     
     In view of Theorem \ref{well-posedness} and Theorem \ref{uniqueness}, we now define the solution semigroup $S(t):\h\rightarrow\h$ associated with problem (\ref{eq}):
     \[S(t)(u_0,u_1):=(u(t),u_t(t)),\]
     where $u(t)$ is the unique weak solution of problem (\ref{eq}) corresponding to the initial data $(u_0,u_1)\in\h$. By Theorem \ref{well-posedness} (ii), the mapping $t\mapsto S(t)\xi$ is continuous from $\mathbb{R}^+$ into $\h$ for any fixed $\xi\in\h$.
     
     \begin{prop}\label{continuity}
     	$\{S(t)\}_{t\geq0}$ is a continuous semigroup, namely, 
     	\[(\xi_n,\tau_n)\rightarrow (\xi,\tau)\ \ \mathrm{strongly}\ \mathrm{in}\ \h\times[0,\infty)\Rightarrow S(\tau_n)\xi_n\rightarrow S(\tau)\xi\ \ \mathrm{strongly}\ \mathrm{in}\ \h.\]
     \end{prop}
     \begin{proof}
     	Let $\tau\geq0$, $\xi\in\h$ with $\n\xi\n_\h=R$, and $(u(t),u_t(t))=S(t)\xi$,  $(u^n(t),u^n_t(t))=S(t)\xi_n$, respectively.
     	
     	\emph{Case 1.} $R=0.$ Then $S(t)\xi=0,\ \forall t\geq0$. Obviously, estimate (\ref{3-2-51}) holds true for $u^n$, then
     	\[\n S(\tau_n)\xi_n-S(\tau)\xi\n_\h=\n(u^n(\tau_n),u^n_t(\tau_n))\n_\h \leq C_\tau\n\xi_n\n_\h\rightarrow C_\tau\n\xi\n_\h=0,\]
     	
     	\emph{Case 2.} $R>0.$ Without loss of generality, $0\leq\tau_n\leq \tau+1$, $\frac{3R}{4}\leq \n\xi_n\n_\h\leq \frac{3R}{2},\ \forall n\geq1$. Since $\xi_n\rightarrow\xi$ strongly in $\h$, $K:=\{\xi_n\}_{n=1}^\infty\cup\{\xi\}$ is compact.
        Then, using (\ref{unique}), we obtain
     	\begin{align*}
     		&\quad\ \n S(\tau_n)\xi_n-S(\tau)\xi\n_\h\\
     		&\leq \n S(\tau_n)\xi_n-S(\tau_n)\xi\n_\h+\n S(\tau_n)\xi-S(\tau)\xi\n_\h\\
     		&\leq C\n\xi_n-\xi\n_\h+\n S(\tau_n)\xi-S(\tau)\xi\n_\h\\
     		&\rightarrow0\ \ \mathrm{as}\ \ n\rightarrow\infty,
     	\end{align*}
     	where $C=C(\tau,R,K)$ and we have used the continuity of the mapping $t\mapsto S(t)\xi$. The proof is complete.
     \end{proof}

\section{Global attractor in $\h$}
     
     In this section, we will verify the existence of the global attractor of the dynamical system $(\h,S(t))$. We recall the difinition of a global attractor here for convenience.
     
     \begin{Def}\cite{Vishik,Hale,Temam}
     	Let $\{S(t)\}_{t\geq0}$ be a semigroup acting on a metric space $(X,d)$. A subset $\mathcal{A}\subset X$ is called a global attractor of $(X,S(t))$ if\\
     	(i) $\mathcal{A}$ is compact in $X$;\\
     	(ii) $\mathcal{A}$ is invariant, i.e. $S(t)\mathcal{A}=\mathcal{A},\ \forall t\geq0$;\\
     	(iii) $\mathcal{A}$ attracts all bounded sets in $X$, i.e. for any bounded set $B\subset X$, 
     	\[dist_X(S(t)B,\mathcal{A}):=\sup_{x\in S(t)B} \inf_{y\in \mathcal{A}}d(x,y)\rightarrow 0,\ as\ t\rightarrow\infty.\]
     \end{Def}
     
     We first verify the dissipativity of $(\h,S(t))$.
     
     \begin{prop}\label{dissipativity}
     	Let Assumption \ref{assumption} be valid. Then the dynamical system $(\h,S(t))$ is dissipative, i.e. there exists a bounded set $\mathcal{B}_0\subset\h$ satisfying: for any bounded set $B\subset\h$, $\exists\ t_B>0$ such that $S(t)B\subset\mathcal{B}_0,\ \forall t\geq t_B.$ In particular, $\mathcal{B}_0$ is called a bounded absorbing set of $(\h,S(t))$.
     \end{prop}
     \begin{proof}
     	Let $B\subset\h$ be an arbitrary bounded set. Due to (\ref{3-7}), there exists some constant $C_B>0$ such that
     	\begin{equation}\label{4-1-1}
     		\n (u(t),u_t(t))\n_\h\leq C_B,\quad\forall t\geq0
     	\end{equation}
     	holds for every weak solution $u$ with its initial data $(u(0),u_t(0))\in B$. Multiplying (\ref{eq}) by $u_t+\alpha u$ with $\alpha>0$ to be determined later, after integrating over $x\in\Omega$, we get
     	\begin{equation}\label{4-1-2}
     		\frac{d}{dt}V(u)+\alpha V(u)+\Gamma=0,
     	\end{equation}
     	where
     	\begin{align*}
     		V(u)=\frac{1}{2}\n u_t\n^2+\frac{\kappa}{2}\n u\n_1^2+\frac{1}{2}\n u\n_2^2+\int_{\Omega}F(u)dx+\alpha(u_t,u),
     	\end{align*}
        and
     	\begin{align*}
     		\Gamma=&\gamma\n(u,u_t)\n_\h^{2q}\n u_t\n_1^2+\frac{\alpha\kappa}{2}\n u\n_1^2+\frac{\alpha}{2}\n u\n_2^2\\
     		&-\frac{3\alpha}{2}\n u_t\n^2-\alpha^2(u_t,u)+\alpha\gamma\n(u,u_t)\n_\h^{2q}(A^\frac{1}{4}u_t,A^\frac{1}{4}u)\\
     		&+\alpha(f(u),u)-\alpha\int_{\Omega}F(u)dx.
     	\end{align*}
     	By (\ref{3-3}), (\ref{3-4}) and (\ref{4-1-1}), there exists $\alpha_0>0$ such that for $\forall \alpha\in(0,\alpha_0]$,
     	\begin{equation}\label{4-1-3}
     		c_1\n(u,u_t)\n_\h^2-C_1\leq V(u)\leq C_B,\quad \forall t\geq0,
     	\end{equation}
     	where $0<c_1<1,\ C_1>0$ are constants. 
     	
     	By Young's inequality with $\varepsilon$ and the Sobolev embedding $\n w\n^2\leq C_\Omega\n w\n^2_1$, we infer that there exists some constant $C_2=C_2(\Omega,q)>0$ such that,
     	\begin{equation}\label{4-1-4}
     		\begin{split}
     			\n u_t\n^2&\leq \frac{1}{C_\Omega}\n u_t\n^{2q+2}+C_2\\
     			&\leq C_\Omega\frac{1}{C_\Omega}\n u_t\n^{2q}\n u_t\n^2_1+C_2\\
     			&\leq \n(u,u_t)\n_\h^{2q}\n u_t\n^2_1+C_2.
     		\end{split}
     	\end{equation}
     	Using Cauchy inequality and Young's inequality, we infer from (\ref{2-4}) and (\ref{4-1-1}) that
     	\begin{equation}\label{4-1-5}
     		\big|\alpha^2(u_t,u)\big|\leq\alpha^2 \n u_t\n\n u\n\leq \frac{\alpha^2}{\sqrt{\lambda_1}}\n u_t\n\n u\n_2\leq \frac{\alpha^3}{2\lambda_1}\n u\n_2^2+\frac{\alpha}{2}\n u_t\n^2,
     	\end{equation}
     	and
     	\begin{equation}\label{4-1-6}
     		\begin{split}
     			\Big|\alpha\gamma\n(u,u_t)\n_\h^{2q}(A^\frac{1}{4}u_t,A^\frac{1}{4}u)\Big|&\leq \alpha\gamma\n(u,u_t)\n_\h^{2q}\n u_t\n_1\n u\n_1\\
     			&\leq \frac{\gamma}{2}\n(u,u_t)\n_\h^{2q}\n u_t\n_1^2+\frac{\alpha^2\gamma}{2}\n(u,u_t)\n_\h^{2q}\n u\n_1^2\\
     			&\leq \frac{\gamma}{2}\n(u,u_t)\n_\h^{2q}\n u_t\n_1^2+\alpha^2C_B\n u\n_2^2.
     		\end{split}
     	\end{equation}
     	From (\ref{2-4}) and the dissipativity condition (\ref{3-5}), we get
     	\begin{equation}\label{4-1-7}
     		\begin{split}
     			\alpha(f(u),u)-\alpha\int_{\Omega}F(u)dx&\geq -\frac{\alpha\mu}{2}\n u\n^2-\alpha C\geq -\frac{\alpha\mu}{2\lambda_1}\n u\n_2^2-\alpha C.
     		\end{split}
     	\end{equation}
     	Thus, it follows from estimates (\ref{4-1-4})-(\ref{4-1-7}) that
     	\begin{equation}\label{4-1-8}
     		\begin{split}
     			\Gamma&\geq \frac{\gamma}{2}\n(u,u_t)\n_\h^{2q}\n u_t\n_1^2+\frac{\alpha}{2}\n u\n_2^2-\frac{\alpha\mu}{2\lambda_1}\n u\n_2^2-\alpha^2C_B\n u\n_2^2\\
     			&\quad-\frac{\alpha^3}{2\lambda_1}\n u\n^2-2\alpha\n(u,u_t)\n_\h^{2q}\n u_t\n^2_1-2\alpha C_2-\alpha C\\
     			&\geq (\frac{\gamma}{2}-2\alpha)\n(u,u_t)\n_\h^{2q}\n u_t\n_1^2+\left(\frac{\lambda_1-\mu}{2\lambda_1}\alpha-\alpha^2C_B-\frac{\alpha^3}{2\lambda_1}\right)\n u\n_2^2\\
     			&\quad-\alpha C_3,
     		\end{split}
     	\end{equation}
         where $C_3=2C_2+C$. 
     	Choose $\alpha\in(0,\alpha_0]$ small enough ($\alpha$ may depend on $B$) such that
     	\[\frac{\gamma}{2}-2\alpha>0,\ \ \mathrm{and}\ \  \frac{\lambda_1-\mu}{2\lambda_1}\alpha-\alpha^2C_B-\frac{\alpha^3}{2\lambda_1}>0.\]
     	Thus, inserting (\ref{4-1-8}) into (\ref{4-1-2}), we obtain
     	\[\frac{d}{dt}V(u)+\alpha V(u)\leq \alpha C_3.\]
     	Applying now Gronwall's lemma and using (\ref{4-1-3}), we end up with
     	\begin{equation}
     		V(u(t))\leq V(u(0))e^{-\alpha t}+C_3(1-e^{-\alpha t})\leq C_Be^{-\alpha t}+C_3.
     	\end{equation}
     	Then there exists $t_B=\max\{0,\frac{1}{\alpha}\ln\frac{C_B}{C_3}\}$ such that
     	\[\n(u(t),u_t(t))\n_\h^2\leq \frac{V(u(t))+C_1}{c_1}\leq \frac{2C_3+C_1}{c_1}=:R_0^2,\quad\forall t\geq t_B.\]
     	Therefore, the dynamical system $(\h,S(t))$ is dissipative and $\mathcal{B}_0=\{\xi\in\h\ |\ \n\xi\n_\h\leq R_0\}$ is a bounded absorbing set.
     \end{proof}
     
     Now, we are in a position to prove our main result of this section.
     
     \begin{thm}\label{attractor}
     	Let Assumption \ref{assumption} be valid. Then the dynamical system $(\h,S(t))$ possesses a global attractor $\mathcal{A}$ in $\h$. Moreover, the global attractor $\mathcal{A}$ is bounded in $\h_s$:
     	\begin{equation}
     		\mathcal{A}\subset\h_s,\quad \n\mathcal{A}\n_{\h_s}:=\sup_{\xi\in\mathcal{A}}\n\xi\n_{\h_s}\leq C,
     	\end{equation}
        where $s=\frac{1}{q}$.
     \end{thm}
     \begin{proof}
     	According to the abstract attractor existence theorem, we only need to verify that $S(t)$ is continuous on $\h$ for every fixed $t\geq0$ and that $(\h,S(t))$ possesses a compact absorbing set in $\h$, see \cite{Hale1,Robinson,Temam}. 
     	
     	The continuity of $S(t)$ comes from Propsition \ref{continuity}. Due to estimate (\ref{3-8}), for any $\xi\in\mathcal{B}_0$, we have
     	\begin{equation}\label{4-3-1}
     		\n S(1)\xi\n_{\h_s}\leq C(\n \xi\n_\h)\leq C_{R_0},
     	\end{equation}
     	where $\mathcal{B}_0=\{\xi\in\h\ |\ \n\xi\n_\h\leq R_0\}$ is the absorbing set constructed in the Proposition \ref{dissipativity} and $s=\frac{1}{q}$. Then, the set
     	\[\mathcal{B}_1:=\{\xi\in\h_s\ |\ \n\xi\n_{\h_s}\leq C_{R_0}\}\]
     	is a compact absorbing set for $(\h,S(t))$ in $\h$. Indeed, $\mathcal{B}_1$	is compact in $\h$ due to the compact embedding $\h_s\hookrightarrow\hookrightarrow\h$. Let $B\subset\h$ be an arbitrary bounded set. Since $\mathcal{B}_0$ is absorbing, there exists $t_B>0$ such that $S(t)B\subset\mathcal{B}_0,\ \forall t\geq t_B$. Then, (\ref{4-3-1}) implies
     	\[S(t)B=S(1)S(t-1)B\subset S(1)\mathcal{B}_0\subset \mathcal{B}_1,\ \forall t\geq t_B+1.\]
     	
     	Thus, the existence of the global attractor $\mathcal{A}$ is proved. Finally, noticing that $\mathcal{A}\subset\mathcal{B}_1$, we have $\mathcal{A}$ is bounded in $\h_s$ and $\n\mathcal{A}\n_{\h_s}\leq C_{R_0}$.
     \end{proof}

\section{Global attractor in $\h_{\frac{1}{q}}$}
     
     We still use the notation $s=\frac{1}{q}\in(0,1]$ throughout this section and will prove that the global attractor $\mathcal{A}$ constructed in Theorem \ref{attractor} is exactly an $(\h,\h_s)$-global attractor.
     
     \begin{Def}\cite{Zhong}
     	Let $X,Y$ be two Banach spaces and $\{S(t)\}_{t\geq0}$ be a semigroup acting on $X$. $\{S(t)\}_{t\geq0}$ is called an $(X,Y)$-semigroup if $S(t)X\subset Y,\ t>0$, and $\{S(t)\}_{t\geq0}$ is called norm-to-weak continuous if additionally,
     	\[(\xi_n,t_n)\rightarrow (\xi,t)\ \ \mathrm{strongly}\ \mathrm{in}\ X\times(0,\infty)\Rightarrow S(t_n)\xi_n\rightarrow S(t)\xi\ \ \mathrm{weakly}\ \mathrm{in}\ Y.\]
     \end{Def}
     
     We have the following result to verify the norm-to-weak continuity of a $(X,Y)$-semigroup.
     
     \begin{lemma}\cite{Zhong}\label{lemma5.2}
     	Let $X,Y$ be two Banach spaces and $X^*,Y^*$ be their dual spaces,  respectively, $\{S(t)\}_{t\geq0}$ be a semigroup on $X$ and an $(X,Y)$-semigroup. Assume that\\
        (i) both $i:Y\rightarrow X$ and $i^*:X^*\rightarrow Y^*$ are densely injective;\\
        (ii) $\{S(t)\}_{t\geq0}$ is continuous or weak continuous on $X$, i.e.
        \[(\xi_n,t_n)\rightarrow (\xi,t)\ \ \mathrm{strongly}\ \mathrm{in}\ X\times(0,\infty)\Rightarrow S(t_n)\xi_n\rightarrow S(t)\xi\ \ \mathrm{strongly}\ \mathrm{in}\ X.\]
        or
        \[(\xi_n,t_n)\rightarrow (\xi,t)\ \ \mathrm{weakly}\ \mathrm{in}\ X\times(0,\infty)\Rightarrow S(t_n)\xi_n\rightarrow S(t)\xi\ \ \mathrm{weakly}\ \mathrm{in}\ X.\]
        
        Then, $\{S(t)\}_{t\geq0}$ is a norm-to-weak continuous $(X,Y)$-semigroup if and only if $\{S(t)\}_{t\geq0}$ maps compact subsets of $Y\times(0,\infty)$ into bounded sets of $Y$.
     \end{lemma}

     The definition of the $(X,Y)$-global attractor is as follows.
     
     \begin{Def}\cite{Babin}
     	A set $\mathcal{A}\subset X\cap Y$ is said to be an $(X,Y)$-global attractor of the $(X,Y)$-semigroup if\\
     	(i) $\mathcal{A}$ is bounded in $X$ and compact in $Y$;\\
     	(ii) $\mathcal{A}$ is invariant, i.e. $S(t)\mathcal{A}=\mathcal{A},\ \forall t\geq0$;\\
     	(iii) $\mathcal{A}$ attracts all bounded subsets of $X$ in the norm topology of $Y$, i.e. for any bounded set $B\subset X$, 
     	\[dist_Y(S(t)B,\mathcal{A}):=\sup_{x\in S(t)B} \inf_{y\in \mathcal{A}}\n x-y\n_Y \rightarrow 0,\ as\ t\rightarrow\infty.\] 
     \end{Def}
     
     In order to verify the existence of global attractors, we use the method of Condition (C) which is first proposeed in \cite{ConditionC}.
     
     \begin{Def}\cite{ConditionC}\label{ConditionC}
     	A semigroup $\{S(t)\}_{t\geq0}$ is said to satisfy Condition (C) in $X$ if and only if for any bounded set $B\subset X$ and for any $\varepsilon>0$, there exist a positive time moment $t_B$ and a finite dimensional subspace $X_1$ of $X$ such that $\{PS(t)x\ |\ x\in B,\ t\geq t_B\}$ is bounded in $X$ and 
     	\[\sup_{x\in B}\n (I-P)S(t)x\n_X\leq\varepsilon,\quad\forall t\geq t_0,\]
     	where $P:X\rightarrow X_1$ is the canonical projector.
     \end{Def}
     
     One of the abstract criteria for the existence of the $(X,Y)$-global attractor of a norm-to-weak continuous $(X,Y)$-semigroup is as follows.
     
     \begin{lemma}\cite{Zhong}\label{lemma5.5}
     	Assume that $X,Y$ are two Banach spaces and $\{S(t)\}_{t\geq0}$ is a norm-to-weak continuous $(X,Y)$-semigroup. Then $\{S(t)\}_{t\geq0}$ possesses an $(X,Y)$-global attractor provided that: \\
     	(i) $\{S(t)\}_{t\geq0}$ has a bounded absorbing set in $Y$, i.e. there exists a bounded set $\mathcal{B}\subset Y$ such that for any bounded set $B\subset X$, $\exists t_B>0$ such that $S(t)B\subset\mathcal{B}$, $\forall t\geq t_B$;\\
     	(ii) $\{S(t)\}_{t\geq0}$ satisfies Condition (C) in $Y$.
     \end{lemma}
     
     By Lemma \ref{lemma5.2}, we infer from Theorem \ref{well-posedness} and Propsition \ref{continuity} that $\{S(t)\}_{t\geq0}$ is actually a norm-to-weak continuous $(\h,\h_s)$-semigroup. Now, we will verify that $\{S(t)\}_{t\geq0}$ satisfies Condition (C) in $\h_s$.
     
     We use the same notations as before. Let $\lambda_1,\lambda_2,\cdots$ be the eigenvalues of the operator $A=\Delta^2$ with boundary condition (\ref{1-3}) or (\ref{1-4}) and $e_1,e_2,\cdots$ be the corresponding eigenfunctions such that
     \[Ae_i=\lambda_i e_i,\quad 0<\lambda_1\leq\lambda_2\leq\cdots,\quad \lim_{i\rightarrow\infty}\lambda_i=\infty,\]
     and $\{e_1,e_2,\cdots\}$ form an orthonormal basis in $L^2$. Let $H_m=\mathrm{span}\{e_1,\cdots,e_m\}$, $P_m:L^2\rightarrow H_m$ be the orthoprojector and $Q_m=I-P_m$ where $I$ is the identity.
     
     \begin{lemma}\label{lemma5}
     	Let $\sigma\in\mathbb{R}$ be fixed and $\mathcal{K}$ be a compact subset of $V_\sigma$. Then for any $\varepsilon>0$ there exists a positive integer $N$ such that
     	\[\n Q_mv\n_\sigma<\varepsilon,\quad \forall m\geq N,\ v\in\mathcal{K}.\]
     \end{lemma}
     \begin{proof}
     	For any $\varepsilon>0$, let $\{v_j\}_{j=1}^M$ be a $\frac{\varepsilon}{2}$-net of $\mathcal{K}$ in $V_\sigma$:
     	\[\mathcal{K}\subset\bigcup_{i=1}^M B_{V_\sigma}(v_j,\frac{\varepsilon}{2}).\]
     	For fixed $i\in\{1,2,\cdots,M\}$, since 
     	\[\n v_j\n_\sigma^2=\n v_j\n_{D(A^\frac{\sigma}{4})}^2=\sum_{i=1}^{\infty}\lambda_i^{\frac{\sigma}{2}}|(v_j,e_i)|^2<\infty,\]
     	there exists an $N_j\in\mathbb{N}^+$ such that 
     	\[\n Q_mv_j\n_\sigma^2=\sum_{i=m+1}^{\infty}\lambda_i^{\frac{\sigma}{2}}|(v_j,e_i)|^2<\frac{\varepsilon}{2}, \quad \forall m\geq N_j.\]
     	Put $N=\max\{N_1,N_2,\cdots,N_M\}$. Then for any $v\in\mathcal{K}$ and $m\geq N$, there exists an $i_0\in\{1,2,\cdots,M\}$ such that $\n v-v_{i_0}\n_\sigma\leq \frac{\varepsilon}{2}$ and 
     	\[\n Q_mv\n_\sigma\leq \n Q_m(v-v_{i_0})\n_\sigma+\n Q_mv_{i_0}\n_\sigma<\varepsilon.\]
     \end{proof}

     \begin{cor}\label{cor5.7}
     	Let Assumption \ref{assumption} be valid. Then the semigroup $\{S(t)\}_{t\geq0}$ satisfies Condition (C) in $\h_s$.
     \end{cor}
     \begin{proof}
     	By Theorem \ref{attractor}, $\mathcal{B}_1$ is an absorbing set of $\{S(t)\}_{t\geq0}$ which is bounded in $\h_s$. Then
     	\[\mathcal{B}:=\bigcup_{t\geq t_*}S(t)\mathcal{B}_1\]
     	is a positively invariant absorbing set bounded in $\h_s$, where $t_*>0$ such that $S(t)\mathcal{B}_1\subset\mathcal{B}_1\ (\forall t\geq t_*)$. Given any $\varepsilon\in(0,1)$, according to Definition \ref{ConditionC}, it suffices to prove that there exist $t_0=t_0(\varepsilon,\mathcal{B})$ and $m_0=m_0(\varepsilon,\mathcal{B})\in\mathbb{N}$ such that 
     	\[\n Q_mS(t)(u_0,u_1)\n_{\h_s}\leq\varepsilon,\quad \forall (u_0,u_1)\in\mathcal{B}\]
     	holds for all $t\geq t_0$ and $m\geq m_0$.
     	
     	For any $(u_0,u_1)\in\mathcal{B}$, set $(u(t),u_t(t))=S(t)(u_0,u_1)$. Since $\mathcal{B}$ is positively invariant bounded in $\h_s$, we have
     	\begin{equation}\label{5-5-1}
     		\n(u(t),u_t(t))\n_{\h_s}\leq C_\mathcal{B},\quad\forall t\geq0.
     	\end{equation}
     	Denoting by $(u^{(1)},u_t^{(1)})=(P_mu,P_mu_t)$ and $(u^{(2)},u_t^{(2)})=(Q_mu,Q_mu_t)$, we need to prove that $\n(u^{(2)}(t),u_t^{(2)}(t))\n_{\h_s}\leq\varepsilon$ for all $t\geq t_0$ and $m\geq m_0$.
     	
     	Let
     	\[W(t)=\frac{1}{2}\n u^{(2)}_t(t)\n_s^2+\frac{\kappa}{2}\n u^{(2)}(t)\n_{1+s}^2+\frac{1}{2}\n u^{(2)}(t)\n_{2+s}^2+\delta(u_t^{(2)}(t),A^\frac{s}{2}u^{(2)}(t))\]
     	where $\delta>0$ small will be determined later. From the embeddings $V_s\hookrightarrow L^2$ and $V_{2+s}\hookrightarrow V_{2s}$, there exists $\delta_0=\delta_0(\Omega)>0$ such that for $\forall\delta\in(0,\delta_0]$,
     	\[\delta\Big|(u_t^{(2)},A^\frac{s}{2}u^{(2)})\Big|\leq \frac{\delta_0}{2}\Big(\n u^{(2)}_t\n^2+\n u^{(2)}\n_{2s}^2\Big)\leq \frac{1}{4}\Big(\n u^{(2)}_t\n_s^2+\n u^{(2)}\n_{2+s}^2\Big)\]
     	and then
     	\begin{equation}\label{5-5-0}
     		\frac{1}{4}\n(u^{(2)}(t),u_t^{(2)}(t))\n_{\h_s}^2\leq W(t)\leq C\n(u^{(2)}(t),u_t^{(2)}(t))\n_{\h_s},
     	\end{equation}
     	where the constant $C=C(\kappa,\Omega)$.
     	
     	Taking $\delta=\min\{\delta_0,1,\frac{\gamma}{4}\varepsilon^q\}$ and multiplying (\ref{eq}) by $A^\frac{s}{2}u^{(2)}_t+\delta A^\frac{s}{2}u^{(2)}$ and integrating over $x\in\Omega$ (which can be justified similarly as in Theorem \ref{well-posedness} (iii)), we get
     	\begin{equation}
     		\begin{split}\label{5-5-2}
     			&\quad\frac{d}{dt}W(t)+2\delta W(t)+\gamma\n(u,u_t)\n_\h^{2q}\n u_t^{(2)}\n_{1+s}^2\\
     			&=2\delta\n u_t^{(2)}\n_s^2+2\delta^2(u_t^{(2)},A^\frac{s}{2}u^{(2)})-\frac{d}{dt}(f(u),A^\frac{s}{2}u^{(2)})+(f'(u)u_t,A^\frac{s}{2}u^{(2)})\\
     			&\quad-\gamma\delta\n(u,u_t)\n_\h^{2q}(A^\frac{1}{2}u_t^{(2)},A^\frac{s}{2}u^{(2)})-\delta(f(u),A^\frac{s}{2}u^{(2)})\\
     			&=2\delta\n u_t^{(2)}\n_s^2+2\delta^2(u_t^{(2)},A^\frac{s}{2}u^{(2)})-\delta(f(u),A^\frac{s}{2}u^{(2)})\\
     			&\quad-\frac{d}{dt}(f(u),A^\frac{s}{2}u^{(2)})+(f'(u)u_t,A^\frac{s}{2}u^{(2)})\\
     			&\quad-\frac{\gamma\delta}{2}\frac{d}{dt}\Big[\n(u,u_t)\n_\h^{2q}\n u^{(2)}\n_{1+s}^2\Big]+\frac{\gamma\delta}{2}\n u^{(2)}\n_{1+s}^2\frac{d}{dt}\n(u,u_t)\n_\h^{2q}.
     		\end{split}
     	\end{equation}
     	
     	By the compact embeddings $V_{2+s}\hookrightarrow\hookrightarrow V_{2s}$, $V_{2+s}\hookrightarrow\hookrightarrow V_{1+s}$ and (\ref{5-5-1}), we deduce from Lemma \ref{lemma5} that there exists $N_1=N_1(\varepsilon)$ such that
     	\begin{equation}\label{5-5-3}
     		\n u^{(2)}\n_{2s}\leq \delta\varepsilon\quad\mathrm{and}\quad\n u^{(2)}\n_{1+s}\leq \delta\varepsilon,\quad\forall t\geq0
     	\end{equation}
     	hold for all $m\geq N_1$. Then, using (\ref{3-3}), (\ref{5-5-1}) and (\ref{5-5-3}), we have
     	\begin{equation}\label{5-5-4}
     		\begin{split}
     			\Big|2\delta^2(u_t^{(2)},A^\frac{s}{2}u^{(2)})-\delta(f(u),A^\frac{s}{2}u^{(2)})\Big|&\leq 2(\n u_t^{(2)}\n+\n f(u)\n)\n u^{(2)}\n_{2s}\\
     			&\leq C(\n u_t\n+\n u\n_2+\n u\n_2^{p+1})\n u^{(2)}\n_{2s}\\
     			&\leq C_\mathcal{B}\delta\varepsilon.
     		\end{split}
     	\end{equation}
     	
     	When the spatial dimension $n\geq5$, since $\frac{4np}{8+4s-nsp}\leq\frac{2n}{n-2(2+s)}$, we have the Sobolev embedding $V_{2+s}\hookrightarrow H^{2+s}\hookrightarrow L^\frac{4np}{8+4s-nsp}$. Combining this with (\ref{3-1}), (\ref{5-5-3}), H\"{o}lder's inequality with $\frac{8+4s-nsp}{4n}+\frac{n-2s}{2n}+\frac{2n-8+nsp}{4n}=1$ and the embeddings $V_s\hookrightarrow H^s\hookrightarrow L^\frac{2n}{n-2s}$, $V_{2-\frac{nsp}{4}}\hookrightarrow H^{2-\frac{nsp}{4}}\hookrightarrow L^\frac{4n}{2n-8+nsp}$, we obtain
     	\begin{equation}\label{5-5-5}
     		\begin{split}
     			&\quad\Big|(f'(u)u_t,A^\frac{s}{2}u^{(2)})\Big|\\
     			&\leq C\n u_t\n\n u^{(2)}\n_{2s}+C\n |u|^p\n_{L^\frac{4n}{8+4s-nsp}}\n u_t\n_{L^\frac{2n}{n-2s}}\n A^\frac{s}{2}u^{(2)}\n_{L^\frac{4n}{2n-8+nsp}}\\
     			&\leq C_\mathcal{B}\delta\varepsilon+C\n u\n^p_{L^\frac{4np}{8+4s-nsp}}\n u_t\n_s\n A^\frac{s}{2}u^{(2)}\n_{2-\frac{nsp}{4}}\\
     			&\leq C_\mathcal{B}\delta\varepsilon+C\n u\n^p_{2+s}\n u_t\n_s\n u^{(2)}\n_{2+2s-\frac{nsp}{4}}\\
     			&\leq C_\mathcal{B}\delta\varepsilon+C_\mathcal{B}\n u^{(2)}\n_{2+2s-\frac{nsp}{4}}.
     		\end{split}
     	\end{equation}
     	Since $2+2s-\frac{nsp}{4}\leq 2+2s-\frac{5s}{4}<2+s$, we have the compact embedding $V_{2+s}\hookrightarrow\hookrightarrow V_{2+2s-\frac{nsp}{4}}$. Then due to Lemma \ref{lemma5}, there exists $N_2=N_2(\varepsilon)$ such that
     	\begin{equation}\label{5-5-6}
     		\n u^{(2)}\n_{2+2s-\frac{nsp}{4}}\leq \delta\varepsilon,\quad\forall t\geq0
     	\end{equation}
     	holds for all $m\geq N_2$. Inserting (\ref{5-5-6}) into (\ref{5-5-5}), we have
     	\begin{equation}\label{5-5-7}
     		\Big|(f'(u)u_t,A^\frac{s}{2}u^{(2)})\Big|\leq C_\mathcal{B}\delta\varepsilon.
     	\end{equation}
     	It is easy to check that (\ref{5-5-7}) still holds when $1\leq n\leq 4$.
     	
     	Moreover, it follows from interpolation inequality (\ref{2-3}) and Young's inequality with $\varepsilon$ that
     	\begin{equation}\label{5-5-8}
     		\n u_t^{(2)}\n_s^2\leq \n u_t^{(2)}\n^\frac{2}{1+s}\n u_t^{(2)}\n_{1+s}^\frac{2s}{1+s}\leq \varepsilon+\varepsilon^{-\frac{1}{s}}\n u_t^{(2)}\n^\frac{2}{s}\n u_t^{(2)}\n_{1+s}^2\leq \varepsilon+\varepsilon^{-q}\n(u,u_t)\n_\h^{2q}\n u_t^{(2)}\n_{1+s}^2
     	\end{equation}
     	
     	Inserting (\ref{5-5-4}), (\ref{5-5-7}) and (\ref{5-5-8}) into (\ref{5-5-2}) and noting that $\delta\leq \frac{\gamma}{4}\varepsilon^q$, we have
     	\begin{equation}\label{5-5-9}
     		\begin{split}
     			\frac{d}{dt}W(t)+2\delta W(t)&\leq-\gamma\n(u,u_t)\n_\h^{2q}\n u_t^{(2)}\n_{1+s}^2+2\delta\varepsilon\\
     			&\quad+2\delta\varepsilon^{-q}\n(u,u_t)\n_\h^{2q}\n u_t^{(2)}\n_{1+s}^2+C_\mathcal{B}\delta\varepsilon-\frac{d}{dt}(f(u),A^\frac{s}{2}u^{(2)})\\
     			&\quad-\frac{\gamma\delta}{2}\frac{d}{dt}\Big[\n(u,u_t)\n_\h^{2q}\n u^{(2)}\n_{1+s}^2\Big]+\frac{\gamma\delta}{2}\n u^{(2)}\n_{1+s}^2\frac{d}{dt}\n(u,u_t)\n_\h^{2q}\\
     			&\leq C_\mathcal{B}\delta\varepsilon-\frac{d}{dt}(f(u),A^\frac{s}{2}u^{(2)})\\
     			&\quad-\frac{\gamma\delta}{2}\frac{d}{dt}\Big[\n(u,u_t)\n_\h^{2q}\n u^{(2)}\n_{1+s}^2\Big]+\frac{\gamma\delta}{2}\n u^{(2)}\n_{1+s}^2\frac{d}{dt}\n(u,u_t)\n_\h^{2q}.
     		\end{split}
     	\end{equation}
     	Applying Gronwall's Lemma, we get
     	\begin{equation}\label{5-5-11}
     		\begin{split}
     			W(t)&\leq W(0)e^{-2\delta t}+C_\mathcal{B}\delta\varepsilon\int_{0}^{t}e^{-2\delta(t-\tau)}d\tau-\int_{0}^{t}e^{-2\delta(t-\tau)}\frac{d}{d\tau}(f(u),A^\frac{s}{2}u^{(2)})d\tau\\
     			&\quad-\frac{\gamma\delta}{2}\int_{0}^{t}e^{-2\delta(t-\tau)}\frac{d}{d\tau}\Big[\n(u,u_t)\n_\h^{2q}\n u^{(2)}\n_{1+s}^2\Big]d\tau\\
     			&\quad+\frac{\gamma\delta}{2}\int_{0}^{t}e^{-2\delta(t-\tau)}\n u^{(2)}\n_{1+s}^2\frac{d}{d\tau}\n(u,u_t)\n_\h^{2q}d\tau.
     		\end{split}
     	\end{equation}
     	Using integration by parts and (\ref{5-5-1}), (\ref{5-5-3}), we have
     	\begin{align*}
     		\quad C_\mathcal{B}\delta\varepsilon\int_{0}^{t}e^{-2\delta(t-\tau)}d\tau
     		=C_\mathcal{B}\delta\varepsilon \cdot \frac{1-e^{-2\delta t}}{2\delta}\leq C_{\mathcal{B}}\varepsilon,
     	\end{align*}
     	\begin{align*}
     		&\quad-\int_{0}^{t}e^{-2\delta(t-\tau)}\frac{d}{d\tau}(f(u),A^\frac{s}{2}u^{(2)})d\tau\\
     		&=-e^{-2\delta(t-\tau)}(f(u(\tau)),A^\frac{s}{2}u^{(2)}(\tau))\Big|_{\tau=0}^{\tau=t}\\
     		&\quad+2\delta\int_{0}^{t}(f(u),A^\frac{s}{2}u^{(2)})e^{-2\delta(t-\tau)}d\tau\\
     		&\leq C_\mathcal{B}\delta\varepsilon+C_\mathcal{B}\delta^2\varepsilon\cdot\frac{1-e^{-2\delta t}}{2\delta}\\
     		&\leq C_{\mathcal{B}}\varepsilon,
     	\end{align*}
     	\begin{align*}
     		&\quad-\frac{\gamma\delta}{2}\int_{0}^{t}e^{-2\delta(t-\tau)}\frac{d}{d\tau}\Big[\n(u,u_t)\n_\h^{2q}\n u^{(2)}\n_{1+s}^2\Big]d\tau\\
     		&=-\frac{\gamma\delta}{2}e^{-2\delta(t-\tau)}\n(u(\tau),u_t(\tau))\n_\h^{2q}\n u^{(2)}(\tau)\n_{1+s}^2\Big|_{\tau=0}^{\tau=t}\\
     		&\quad+\gamma\delta^2\int_{0}^{t}\n(u,u_t)\n_\h^{2q}\n u^{(2)}\n_{1+s}^2e^{-2\delta(t-\tau)}d\tau\\
     		&\leq C_\mathcal{B}\delta^3\varepsilon^2+C_\mathcal{B}\delta^4\varepsilon^2\cdot\frac{1-e^{-2\delta t}}{2\delta}\\
     		&\leq C_{\mathcal{B}}\varepsilon.
     	\end{align*}
        With regard to the last term in (\ref{5-5-11}), it follows from the energy equality (\ref{energy}) that
        \[\gamma\int_{0}^{t}\n(u,u_t\n_\h^{2q})\n u_t\n_1^2d\tau=E(u(0))-E(u(t))\leq C_\mathcal{B},\quad\forall t\geq0.\]
        Then, since $u$ solves the equation (\ref{eq}), we have
     	\begin{align*}
     		&\quad\frac{\gamma\delta}{2}\int_{0}^{t}e^{-2\delta(t-\tau)}\n u^{(2)}\n_{1+s}^2\frac{d}{d\tau}\n(u,u_t)\n_\h^{2q}d\tau\\
     		&=\frac{q\gamma\delta}{2}\int_{0}^{t}e^{-2\delta(t-\tau)}\n u^{(2)}\n_{1+s}^2\n(u,u_t)\n_\h^{2(q-1)}\frac{d}{d\tau}\Big[\n u_t\n^2+\n u\n_2^2\Big]d\tau\\
     		&=\frac{q\gamma\delta}{2}\int_{0}^{t}e^{-2\delta(t-\tau)}\n u^{(2)}\n_{1+s}^2\n(u,u_t)\n_\h^{2(q-1)}\Big(\kappa A^\frac{1}{2}u+\gamma\n(u,u_t)\n_\h^{2q}A^\frac{1}{2}u_t+f(u),u_t\Big)d\tau\\
     		&\leq C_\mathcal{B}\delta^3\varepsilon^2\cdot\frac{1-e^{-2\delta t}}{2\delta}+C_\mathcal{B}\delta^3\varepsilon^2\int_{0}^{t}\n(u,u_t\n_\h^{2q})\n u_t\n_1^2d\tau\\
     		&\leq C_{\mathcal{B}}\varepsilon.
     	\end{align*}
     	Plugging these estimates into (\ref{5-5-11}) and using (\ref{5-5-0}), we conclude that
     	\begin{align*}
     		\n(u^{(2)}(t),u_t^{(2)}(t))\n_{\h_s}^2\leq4W(t)\leq C\n(u_0,u_1)\n_{\h_s}^2e^{-2\delta t}+C_\mathcal{B}\varepsilon\leq 2C_\mathcal{B}
     		\varepsilon
     	\end{align*}
     	provided that  $t\geq\frac{1}{2\delta}\ln\frac{1}{\varepsilon}=(\min\{\delta_0,1,\frac{\gamma}{4}\varepsilon^q\})^{-1}\ln\frac{1}{\sqrt{\varepsilon}}$ and $m\geq\max\{N_1,N_2\}$. The proof is complete.
     \end{proof}
     
     Thanks to Lemma \ref{lemma5.5}, using Theorem \ref{attractor} and Corollary \ref{cor5.7}, we have obtained the following result:
     
     \begin{thm}\label{thm}
     	Let Assumption \ref{assumption} be valid. Then the semigroup $\{S(t)\}_{t\geq0}$ generated by weak solutions of Eq. (\ref{eq}) is a norm-to-weak continuous $(\h,\h_s)$-semigroup and possesses an $(\h,\h_s)$-global attractor $\mathcal{A}_s$, that is, $\mathcal{A}_s$ is invariant, compact in $\h_s$ and attracts any bounded subset of $\h$ in the norm topology of $\h_s$. 
     	
     	Moreover, according to the definition of attractors, it is obvious that $\mathcal{A}=\mathcal{A}_s$, where $\mathcal{A}$ is the global attractor of dynamical system $(\h,S(t))$ constructed in Theorem \ref{attractor}.
     \end{thm}
     
     \begin{remark}
     	Theorem \ref{attractor} implies that the dynamical system $(\h,S(t))$ posesses a compcat absorbing set, then by \citep[Theorem 3.1]{Zhang}, $\{S(t)\}_{t\geq0}$ is global exponentially $\kappa$-dissipative and there exists an exponential attraction set of $(\h,S(t))$, i.e. there exists a compact subset $\mathcal{A}^*\subset\h$ such that $\mathcal{A}^*$ is positively invariant and exponentially attracts any bounded subset $B\subset\h$:
     	\[dist_\h(S(t)B,\mathcal{A}^*)\leq C(\n B\n_\h)e^{-\alpha t},\quad\forall t\geq0.\]
     	
     	Note that $\mathcal{A}^*$ is not an exponential attractor of $(\h,S(t))$ since we don't know whether it has finite fractal dimension. In fact, to our best knowledge, there are no any criteria to verify the finiteness of fractal dimension of global attractors in the degenerate case like problem (\ref{1-1}). We shall concern this interesting problem in the future.
     \end{remark}

  \section*{Acknowledgements}
  This work is supported by the National Science Foundation of China Grant (11731005).
  
  \section*{Data Availability Statement}	
  Data sharing is not applicable to this article as no new data were created or analyzed in this study.

\end{document}